\documentclass[11pt,reqno]{amsart}
\usepackage[a4paper,margin=1.02in]{geometry}
\usepackage{amsmath,amssymb,amsthm,mathtools,comment}
\usepackage{empheq}
\usepackage{bm}
\usepackage{booktabs,array,tabularx}
\usepackage{float}
\usepackage{aliascnt}
\usepackage{enumitem}
\usepackage{microtype}
\usepackage[hidelinks]{hyperref}
\allowdisplaybreaks
\setlist[enumerate]{leftmargin=2.4em,itemsep=0.25em,topsep=0.35em}
\setlist[itemize]{leftmargin=2.1em,itemsep=0.25em,topsep=0.35em}
\numberwithin{equation}{section}

\newtheorem{theorem}{Theorem}[section]

\newaliascnt{proposition}{theorem}
\newtheorem{proposition}[proposition]{Proposition}
\aliascntresetthe{proposition}

\newaliascnt{corollary}{theorem}
\newtheorem{corollary}[corollary]{Corollary}
\aliascntresetthe{corollary}

\newaliascnt{lemma}{theorem}
\newtheorem{lemma}[lemma]{Lemma}
\aliascntresetthe{lemma}

\newaliascnt{assumption}{theorem}

\aliascntresetthe{assumption}

\newaliascnt{hypothesis}{theorem}

\aliascntresetthe{hypothesis}

\theoremstyle{definition}
\newaliascnt{definition}{theorem}

\aliascntresetthe{definition}

\newaliascnt{remark}{theorem}
\newtheorem{remark}[remark]{Remark}
\aliascntresetthe{remark}

\newaliascnt{example}{theorem}

\aliascntresetthe{example}

\usepackage[nameinlink,noabbrev]{cleveref}
\crefname{theorem}{Theorem}{Theorems}
\Crefname{theorem}{Theorem}{Theorems}
\crefname{proposition}{Proposition}{Propositions}
\Crefname{proposition}{Proposition}{Propositions}
\crefname{corollary}{Corollary}{Corollaries}
\Crefname{corollary}{Corollary}{Corollaries}
\crefname{lemma}{Lemma}{Lemmas}
\Crefname{lemma}{Lemma}{Lemmas}
\crefname{assumption}{Assumption}{Assumptions}
\Crefname{assumption}{Assumption}{Assumptions}
\crefname{hypothesis}{Standing hypothesis}{Standing hypotheses}
\Crefname{hypothesis}{Standing hypothesis}{Standing hypotheses}
\crefname{definition}{Definition}{Definitions}
\Crefname{definition}{Definition}{Definitions}
\crefname{remark}{Remark}{Remarks}
\Crefname{remark}{Remark}{Remarks}
\crefname{example}{Example}{Examples}
\Crefname{example}{Example}{Examples}

\newcommand{\R}{\mathbb R}
\newcommand{\Rplus}{\mathbb R_+}
\newcommand{\E}{\mathbb E}
\newcommand{\Pp}{\mathbb P}
\newcommand{\1}{\mathbf 1}
\newcommand{\cL}{\mathcal L}
\newcommand{\dd}{\,\mathrm d}
\newcommand{\e}{\mathrm e}
\newcommand{\tauX}{\tau_0^X}
\newcommand{\tauY}{\tau_0^Y}
\newcommand{\Xlow}{\overline X}
\newcommand{\doi}[1]{\href{https://doi.org/#1}{doi:#1}}
\newcommand{\rev}[1]{#1}
\newcommand{\vnew}[1]{#1}

\newcommand{\vthirtyfour}[1]{#1}
\def\beqlb{\begin{eqnarray}}\def\eeqlb{\end{eqnarray}}
\def\beqnn{\begin{eqnarray*}}\def\eeqnn{\end{eqnarray*}}
\def\ar{\!\!\!&}

\title[Nonlinear predator-prey branching model]
{Extinction and extinguishment properties for a nonlinear predator-prey branching model}

\author{Lina Ji}
\address{MSU-BIT-SMBU Joint Research Center of Applied Mathematics, Shenzhen MSU-BIT University, Shenzhen, China.}
\email{jiln@smbu.edu.cn}

\author{Jie Xiong}
\address{Department of Mathematics and Shenzhen International Center for Mathematics,
	Southern University of Science and Technology, Shenzhen 518055, China.}
\email{xiongj@sustech.edu.cn}

\author{Wen Xu }
\address{School of Mathematical Sciences, Peking University, Beijing 100871, China.}
\email{xuwen@math.pku.edu.cn}

\author{Xu Yang}
\address{School of Mathematics and Information Science, North Minzu University, Yinchuan, China.}
\email{xuyang@mail.bnu.edu.cn}

\thanks{This work was supported by
	the National Key R\&D Program of China (2022YFA1006102),
	the National Natural Science Foundation of China (12471418, 12595294, 12231002, 12301167, 12471135), the New Cornerstone Science Foundation (NCI202501), Guangdong Basic and Applied Basic Research Foundation  (2022A1515110986), and Shenzhen National Science Foundation  (20231128093607001).}

\date{}

\hypersetup{
  pdftitle={\vthirtyfour{Extinction and Extinguishment Properties for a Nonlinear Predator-Prey Branching Model}},
  pdfauthor={Lina Ji, Jie Xiong, Wen Xu and Xu Yang},
  pdfsubject={Extinction and extinguishment in mixed-sign continuous-state population dynamics},
  pdfkeywords={continuous-state branching process, mixed-sign interaction, stable \vthirtyfour{L\'evy} noise, nonexplosion, extinction, extinguishment}
}

\begin{document}

\begin{abstract}
	We study extinction and extinguishment in a two-type continuous-state nonlinear branching model driven by Brownian branching noise and spectrally positive stable jumps.  The populations are subject to nonlinear self-regulation and a mixed-sign predator--prey interaction: the second promotes the first, whereas the first suppresses the second. Two complementary structures are developed.  An exact power--logarithmic cancellation functional removes the interaction drifts and yields stochastic Lyapunov estimates, nonexplosion, and boundary criteria.  In the multiplicative regimes, integrating-factor identities and geometric L\'evy factorizations express extinction through weighted exposure clocks and reduce the long-time analysis to effective decay rates.  These methods yield almost-sure extinction criteria and identify a regime in which finite-time extinction and nonextinction coexist. On nonextinction, both populations remain positive at all finite times and converge jointly to zero, exhibiting joint extinguishment rather than positive persistence.
\end{abstract}


\maketitle

\section{Introduction and main results}
\label{sec:introduction}

 Continuous-state branching processes (CSBPs), introduced by
Ji\v{r}ina~\cite{Jirina1958}, are the continuous-mass analogues of
Galton--Watson processes. They arise as scaling limits of suitably rescaled Galton--Watson processes,
as shown by Lamperti~\cite{Lamperti1967CSBP,Lamperti1967Limit} and
Grimvall~\cite{Grimvall1974}. In the conservative finite-mean setting, a CSBP $Z := (Z_t)_{t \ge 0}$ can be represented as
the solution of the Dawson--Li equation~\cite{DawsonLi2012},
\beqnn
	Z_t= z-\beta\int_0^t Z_s\,\dd s
	+\int_0^t\sqrt{2cZ_s}\,\dd B_s
	+\int_0^t\int_0^\infty\int_0^{Z_{s-}}
	r\,\widetilde N(\dd s,\dd r,\dd u),
\eeqnn
where \(c\ge0\), \(\beta\in\mathbb R\), \((B_t)_{t \ge 0}\) is a Brownian motion, and
\(\widetilde N\) is an independent compensated Poisson random measure. For a systematic account of CSBPs, the reader is referred to Li~\cite{Li2020CBI,Li2022} and Pardoux~\cite{Pardoux2016}.   A central question in the classical theory is whether the absorbing state zero
is reached in finite time.  Grey~\cite{Grey1974} established the corresponding
integral criterion.  When this condition holds, a critical or subcritical CSBP
becomes extinct in finite time almost surely, whereas a supercritical CSBP
survives with positive probability.

Nonlinear branching models generalize the classical branching theory by allowing the birth and death rates to depend on the current population mass.   The basic example is the logistic CSBP introduced by
Lambert~\cite{Lambert2005}, obtained by adding quadratic competition.  Under suitable conditions, Lambert proved that the process converges to zero
	almost surely and reaches zero in finite time if and only if Grey's condition
	holds. Cattiaux et
al.~\cite{CattiauxEtAl2009} studied quasi-stationarity for the logistic Feller
diffusion, and Foucart~\cite{Foucart2019} classified the boundary at infinity
for logistic CSBPs. Li et al.~\cite{LiYangZhou2019} developed a
general theory of nonlinear branching SDEs, establishing criteria for
extinction, explosion, and coming down from infinity.  Ma et
al.~\cite{MaYangZhou2021} subsequently treated the critical power-rate cases
left open by this analysis.

Interacting branching systems describe populations whose branching dynamics
are coupled through their current masses.
Ren et al.~\cite{RenXiongYangZhou2022} analyzed a one-way model in which an
autonomous population suppresses a second one, obtaining nearly sharp
extinction--extinguishment criteria.  For two-way systems, Cattiaux and
M\'el\'eard~\cite{CattiauxMeleard2010} studied competitive and weakly
cooperative logistic Feller diffusions, while Xiong et
al.~\cite{XiongYangZhouCompeting,XiongYangZhouEnhancing} treated mutual
competition and mutual enhancement in models with Brownian and stable noises.
The present paper addresses a two-way branching system with
mixed-sign interaction, a regime not covered by the preceding models.

 For \(i=1,2\), let \(B_i=(B_i(t))_{t\ge0}\) be a Brownian motion
and let \( N_i(\dd s,\dd z,\dd u)\) be a Poisson random measure on
$\mathbb{R}_+\times(0,\infty)^2$ with intensity
$\dd s\,\mu_i(\dd z)\,\dd u$, where 
\beqnn 
\mu_i(\dd z)=\frac{\alpha_i(\alpha_i-1)}
{\Gamma(\alpha_i)\Gamma(2-\alpha_i)}z^{-1-\alpha_i}
\1_{\{z>0\}}\dd z, \qquad \alpha_i\in(1,2).
\eeqnn 
Let $\widetilde N_i$ denote the corresponding compensated measure.  Assume that $B_1,B_2,N_1, N_2$ are mutually independent, and let
	\((\mathcal F_t)_{t\ge0}\) be the usual augmentation of the filtration
	generated by these driving noises. For each
\((x,y)\in(0,\infty)^2\), we consider the following system:
\begin{empheq}[left=\empheqlbrace]{equation}
\begin{aligned}
 X_t={}&x+
 \int_0^t\bigl(a_1 X_s^{\theta_1}Y_s^{\kappa_1}-b_{10}X_s^{r_{10}}\bigr)\dd s
 +\int_0^t\sqrt{2b_{11}X_s^{r_{11}}}\dd B_1(s)
 \\
 &+\int_0^t\int_0^\infty\int_0^{b_{12}X_{s-}^{r_{12}}}
 z\,\widetilde N_1(\dd s,\dd z,\dd u),\\
 Y_t={}&y+
 \int_0^t\bigl(-a_2 Y_s^{\theta_2}X_s^{\kappa_2}-b_{20}Y_s^{r_{20}}\bigr)\dd s
 +\int_0^t\sqrt{2b_{21}Y_s^{r_{21}}}\dd B_2(s)
 \\
 &+\int_0^t\int_0^\infty\int_0^{b_{22}Y_{s-}^{r_{22}}}
 z\,\widetilde N_2(\dd s,\dd z,\dd u).
\end{aligned}
 \label{eq:noimm-system}
\end{empheq}
Here \(a_i,\kappa_i>0\), \(\theta_i\ge0\), and
\(r_{ij},b_{ij}\ge0\) for \(i=1,2\) and \(j=0,1,2\).
By localization, the system has a pathwise unique \vthirtyfour{strong}
c\`adl\`ag solution up to its first boundary or explosion time; see Lemma
\ref{prop:local-wellposedness}.  We denote its law by \(\Pp_{x,y}\) and the
corresponding expectation by \(\E_{x,y}\).

The system may be interpreted as a rabbit--grass model, with
\(X\) representing rabbits and \(Y\) grass: \(Y\) promotes \(X\), whereas
\(X\) suppresses \(Y\).  The extinction of \(Y\) depends not only on its
autonomous dynamics but also on its cumulative exposure to \(X\).  Moreover,
the mixed-sign interaction precludes a simultaneous one-sided comparison of
both coordinates with their autonomous counterparts.

More general nonlinear predator--prey branching systems can be
obtained by replacing the two power-type interaction rates by suitable
nonnegative continuous functions \(f_i(x,y)\), the power-type branching
coefficients by general nonnegative functions \(\gamma_{ij}\), and the stable
L\'evy measures by more general spectrally positive L\'evy measures satisfying
the usual integrability conditions.  We restrict attention to the
power-type coefficients and stable L\'evy densities in
\eqref{eq:noimm-system} in order to make the underlying mechanisms transparent
and the resulting criteria explicit.

For \(R>0\), set
\[
 \tau_R^+:=\inf\{t\ge0:X_t\vee Y_t\ge R\},
 \qquad
 \tau_R^-:=\inf\{t\ge0:X_t\wedge Y_t\le R\},
\]
and define the explosion and boundary times, respectively, by
\[
 \tau_\infty:=\lim_{R\uparrow\infty}\tau_R^+,
 \qquad
 \tau_0:=\lim_{R\downarrow0}\tau_R^-.
\]
Since the coordinates have only nonnegative jumps, on
\(\{\tau_0<\tau_\infty\}\) the limit
\((X_{\tau_0},Y_{\tau_0}):=\lim_{t\uparrow\tau_0}(X_t,Y_t)\) exists in
\([0,\infty)^2\) and has at least one zero coordinate.

We call the model $(X, Y)$ in \eqref{eq:noimm-system} \emph{source-conservative} if
\begin{equation}
 \Pp_{x,y}\bigl(\tau_\infty<\infty,\ \tau_\infty\le\tau_0\bigr)=0
 \rev{,\qquad x,y>0.}
 \label{eq:source-conservative}
\end{equation}
When \eqref{eq:source-conservative} holds, we freeze the system at its first
boundary time:
\beqnn 
 (X_t,Y_t)=(X_{\tau_0},Y_{\tau_0}),
 \qquad t\ge\tau_0;
\eeqnn
equivalently, all integrals in \eqref{eq:noimm-system} are stopped at
\(\tau_0\).  The resulting source-stopped process is strong
Markov.

For this source-stopped process, define
\[
 \tauX:=\inf\{t\ge0:X_t=0\},
 \qquad
 \tauY:=\inf\{t\ge0:Y_t=0\}.
\]
Then \(\tau_0=\tauX\wedge\tauY\); if one coordinate reaches zero first, the
other is frozen and its own hitting time is \(\infty\).  We call a coordinate
\emph{extinct} when its hitting time is finite, and \emph{extinguished} when
its hitting time is infinite but it converges to zero.  Unless stated
otherwise, all global results concern the source-stopped process, and
``conservative'' means source-conservative in the sense of
\eqref{eq:source-conservative}.

To state the boundary and long-time results, we introduce the
following effective parameters.  With the convention
\(\min\varnothing=+\infty\), define 
\begin{equation}
 r_i:=\min\Bigl(
 \{r_{i0}-1:b_{i0}>0\}
 \cup\{r_{i1}-2:b_{i1}>0\}
 \cup\{r_{i2}-\alpha_i:b_{i2}>0\}
 \Bigr),
 \qquad i=1,2. 
 \label{eq:effective-minima}
\end{equation}
 Also set
\beqnn 
 b_i:=b_{i0}\1_{\{r_{i0}-1=r_i\}}
     +b_{i1}\1_{\{r_{i1}-2=r_i\}}
     +b_{i2}\1_{\{r_{i2}-\alpha_i=r_i\}},
 \qquad i=1,2.
\eeqnn
Thus \(b_i\) is the sum of the coefficients corresponding to
the minimum in \eqref{eq:effective-minima}. 

 We investigate whether the enhancing interaction can cause
explosion before the first boundary time and how the mixed-sign feedback
affects extinction.  A power--logarithmic functional \(H\) cancels the two
interaction drifts exactly.  This yields source-conservativeness in both the
concave and dissipative regimes, together with quantitative energy and
occupation estimates, respectively
(Theorems~\ref{thm:concave-cancellation}
and~\ref{thm:dissipativity}).  Theorem~\ref{thm:boundary-clock} gives boundary
nonattainment criteria for both coordinates and, when
\(0\le\theta_2<1\), shows that divergent cumulative exposure to \(X\) forces
\(Y\) to reach zero in finite time.  Theorem~\ref{thm:kappa-le-r} complements
this pathwise criterion by identifying an explicit parameter regime for
almost-sure extinction of \(Y\).  Since zero is inaccessible for each
autonomous coordinate in this regime, the extinction of \(Y\) is induced by
the suppressive interaction. 

 In the multiplicative regime for \(Y\), an integrating-factor
identity gives an exact weighted-clock characterization of \(\tauY\)
(Theorem~\ref{thm:FK-rate}).  In the fully multiplicative model, the
corresponding factorization of \(X\) also yields source-conservativeness and
leads to an effective-rate transition.  One rate ordering gives almost-sure
extinction of \(Y\), including at the critical equality, whereas under the
strict reverse ordering
\(\Pp_{x,y}(\tauY<\infty)\to0\) as \(x\downarrow0\) for every fixed \(y>0\).
When the \(X\)-noise is nondegenerate, extinction nevertheless has positive
probability for every \(x,y>0\); hence extinction and nonextinction coexist
for all sufficiently small \(x\).  On \(\{\tauY=\infty\}\), both coordinates
remain strictly positive at every finite time but converge to zero as
\(t\to\infty\).  Their logarithmic decay is quantified in
Theorem~\ref{thm:critical-transition}. 

 The proofs rely on three structural reductions.  First, the function \(H\) is both a first integral of the deterministic interaction
		system and a stochastic Lyapunov function.  Stable homogeneity makes the
		action of the operator \(\cL\), defined in
		\eqref{eq:noimm-generator}, explicit.  Localized It\^o estimates then yield
		the energy and occupation bounds in
		Theorems~\ref{thm:concave-cancellation} and~\ref{thm:dissipativity},
		respectively.  Conditional increment estimates and a conditional
		Borel--Cantelli argument further convert finite occupation into joint
		convergence to zero on \(\{\tau_0=\infty\}\).
 The second reduction uses inverse-power test functions to
		establish the boundary nonattainment criteria in
		Theorem~\ref{thm:boundary-clock}\textup{(i)--(ii)}, while the transform
		\(Y^{1-\theta_2}\) yields the exposure estimate in part~\textup{(iii)}.
		A scalar lower comparison process for \(X\), together with comparison with an
		associated one-way system, gives the explicit extinction criterion in
		Theorem~\ref{thm:kappa-le-r}.
 Finally, in the multiplicative regimes, a relative-jump
		representation yields independent geometric L\'evy factors.  An
		integrating-factor calculation gives the weighted-clock identity in
		Theorem~\ref{thm:FK-rate}.  In the fully multiplicative model, the
		corresponding factorization of \(X\) reduces the transition and decay results 
		in Theorem~\ref{thm:critical-transition} to exponential functionals of L\'evy
		processes.   The unbounded-support estimate in Lemma~\ref{lem:clock-unbounded-support} yields positive extinction probability
		in the reverse regime when the \(X\)-noise is nondegenerate.

		\paragraph{\bf{Notation.}}
 \(C^2((0,\infty)^2)\) denotes the space of twice continuously
 differentiable real-valued functions on \((0,\infty)^2\).  For real-valued \(f\) and eventually positive \(g\),
 \(f=o(g)\) means that \(f/g\to0\).  For nonnegative \(f\) and \(g\), with
 \(g>0\) eventually, \(f\sim g\) means that \(f/g\to1\), and \(f\asymp g\)
 means that \(cg\le f\le Cg\) eventually for some \(0<c\le C<\infty\).
 In particular, 
 \(o(1)\) denotes a quantity tending to zero. 

\subsection{Cancellation and Lyapunov results}\label{subsection1.1}

For \(f\in C^2((0,\infty)^2)\), we define \(\cL f\) by
\begin{align}
 \cL f(x,y)={}&
 \bigl(a_1x^{\theta_1}y^{\kappa_1}-b_{10}x^{r_{10}}\bigr)f_x'(x,y)
 +\bigl(-a_2y^{\theta_2}x^{\kappa_2}-b_{20}y^{r_{20}}\bigr)f_y'(x,y)
 \notag\\
 &+b_{11}x^{r_{11}}f_{xx}''(x,y)+b_{21}y^{r_{21}}f_{yy}''(x,y)
 \notag\\
 &+b_{12}x^{r_{12}}\int_0^\infty
 \bigl[f(x+z,y)-f(x,y)-zf_x'(x,y)\bigr]\mu_1(\dd z)
 \notag\\
 &+b_{22}y^{r_{22}}\int_0^\infty
 \bigl[f(x,y+z)-f(x,y)-zf_y'(x,y)\bigr]\mu_2(\dd z)
 \label{eq:noimm-generator}
\end{align}
whenever the integrals on the right-hand side are absolutely convergent and locally bounded. Subscripts denote the corresponding first- and second-order partial
derivatives. When \(\cL f\) is bounded on
\([\epsilon,R]^2\), by It\^o's formula,
\beqnn
f(X_{t\wedge\tau_R^+\wedge\tau_\epsilon^-},
Y_{t\wedge\tau_R^+\wedge\tau_\epsilon^-})-f(x,y)-\int_0^{t\wedge\tau_R^+\wedge\tau_\epsilon^-}
\cL f(X_s,Y_s)\dd s
\eeqnn
is a local martingale. If, in addition, $f$ is nonegative, by Fatou's lemma and the localizing sequence of stopping times for a local martingale, we obtain 
	\beqlb\label{mart}
	 \mathbb{E}_{x,y}\left[f(X_{t\wedge\tau_R^+\wedge\tau_\epsilon^-},
	Y_{t\wedge\tau_R^+\wedge\tau_\epsilon^-})\right] \le f(x, y) + \mathbb{E}_{x,y}\left[\int_0^{t\wedge\tau_R^+\wedge\tau_\epsilon^-}
	\cL f(X_s,Y_s)\dd s \right].
	\eeqlb

Write \(q_1=1+\kappa_2-\theta_1\) and
\(q_2=1+\kappa_1-\theta_2\).  For \(q\in\mathbb R\), let
\(\Phi_q:(0,\infty)\to\mathbb R\) be defined by 
\[
\Phi_q(u):=
\begin{cases}
u^q/q,&q\ne0,\\
\log u,&q=0,
\end{cases}.
\]
When \(q>0\), we use its continuous extension
\(\Phi_q(0):=0\).  Set 
\begin{equation}
H(x,y):=\Phi_{q_1}(x)+\frac{a_1}{a_2}\Phi_{q_2}(y).
\label{eq:H}
\end{equation} 
Then the mixed interaction cancels exactly:
\begin{equation}
 a_1x^{\theta_1}y^{\kappa_1}\partial_xH
 -a_2y^{\theta_2}x^{\kappa_2}\partial_yH=0, \qquad x, y > 0.
 \label{eq:interaction-cancel}
\end{equation}
 The above~\eqref{eq:interaction-cancel} shows that \(H\) is
conserved by the deterministic interaction dynamics and also motivates its
use as a Lyapunov functional below.  The corresponding deterministic phase
portrait is given in Proposition~\ref{prop:ODE-classification}. 

For \(\alpha\in(1,2)\) and \(q<\alpha\), set
\begin{equation}
 \rev{j_\alpha(q)
 :=\frac{(q-1)\Gamma(\alpha-q)}
        {\Gamma(\alpha)\Gamma(2-q)}.
 }
 \label{eq:j-alpha}
\end{equation}
\rev{For \(i=1,2\) and \(0<q<\alpha_i\), define}
\begin{equation}
 D_i^{(q)}(z):=b_{i0}z^{q+r_{i0}-1}
 +(1-q)b_{i1}z^{q+r_{i1}-2}
 -b_{i2}j_{\alpha_i}(q)z^{q+r_{i2}-\alpha_i},
 \qquad z>0.
 \label{eq:Dip}
\end{equation}
 Whenever \(0<q_i<\alpha_i\) for \(i=1,2\),
\eqref{eq:noimm-generator}, \eqref{eq:interaction-cancel}, and
Lemma~\ref{lem:stable-power} yield 
\begin{equation}
 \rev{\cL H(x,y)
 =-D_1^{(q_1)}(x)-\frac{a_1}{a_2}D_2^{(q_2)}(y).}
 \label{eq:LH-D}
\end{equation}

 When \(0<q_i\le1\), \(i=1,2\), the functions
\(\Phi_{q_i}\) are concave, and hence \(D_i^{(q_i)}\ge0\).  Thus
\eqref{eq:LH-D} gives \(\cL H\le0\), leading to the following Lyapunov
estimate. 

\begin{theorem} 
\label{thm:concave-cancellation}
 Assume that \(0<q_1,q_2\le1\).  Then the model
\eqref{eq:noimm-system} is source-conservative.  Moreover, for the associated
source-stopped process and every \(t\ge0\), 
\begin{align}
  \E_{x,y}H(X_t,Y_t)
 +\E_{x,y}\int_0^{t\wedge\tau_0}
 \left[D_1^{(q_1)}(X_s)+\frac{a_1}{a_2}D_2^{(q_2)}(Y_s)\right]\dd s
  \le H(x,y).
 \label{eq:concave-energy}
\end{align}
Consequently,
\(\bigl(H(X_t,Y_t)\bigr)_{t\ge0}\) is a nonnegative supermartingale.  If, in
addition, \(b_{10},b_{20}>0\), then, almost surely,
\beqnn
\lim_{t\to\infty}(X_t,Y_t)
=
\begin{cases}
	(X_{\tau_0},Y_{\tau_0}), & \tau_0<\infty,\\
	(0,0), & \tau_0=\infty.
\end{cases}
\eeqnn 
\end{theorem}

 If \(q_i>1\) for some \(i\), the corresponding diffusion and
jump contributions to \(\cL H\) may be positive.  The next result imposes a
dominance condition under which the self-regulating drifts control these
contributions at infinity, yielding a coercive bound for \(\cL H\). 

\begin{theorem}
\label{thm:dissipativity}
Assume \(0<q_i<\alpha_i\) and \(b_{i0}>0\), \(i=1,2\), and
\begin{equation}
 r_{i0}-1>
 \max\Bigl(
 \{0\}\cup\{r_{i1}-2:b_{i1}>0\}
 \cup\{r_{i2}-\alpha_i:b_{i2}>0\}
 \Bigr),
 \qquad i=1,2.
 \label{eq:dissipative-b}
\end{equation}
Whenever \(q_i>1\), assume additionally that
\begin{equation}
 q_i+r_{i1}-2\ge0\quad\text{if }b_{i1}>0,
 \qquad
 q_i+r_{i2}-\alpha_i\ge0\quad\text{if }b_{i2}>0.
 \label{eq:dissipative-c}
\end{equation}
\rev{Then there exist constants $C,c>0$ such that}
\beqnn
 \cL H(x,y)
 \le C-c\left(x^{q_1+r_{10}-1}+y^{q_2+r_{20}-1}\right),
 \qquad x,y>0.
\eeqnn 
 The model \eqref{eq:noimm-system} is
source-conservative and, for every \(T>0\), 
\beqnn
 \frac1T\E_{x,y}\int_0^{T\wedge\tau_0}
 \left(X_s^{q_1+r_{10}-1}+Y_s^{q_2+r_{20}-1}\right)\dd s
 \le \frac Cc+\frac{H(x,y)}{cT}.
 \eeqnn 
\end{theorem}

\subsection{Boundary behavior and extinction}
 We next turn to the accessibility of the coordinate axes.  The
following theorem gives boundary nonattainment criteria for both coordinates
and shows that, when \(0\le\theta_2<1\), divergent cumulative exposure to
\(X\) forces \(Y\) to become extinct in finite time. 

\begin{theorem} 
\label{thm:boundary-clock}
Assume that the model $(X, Y)$ in \eqref{eq:noimm-system} is
source-conservative.
\begin{enumerate}[label=\textup{(\roman*)}]
\item If \(r_1\ge0\), then \(\Pp_{x,y}(\tauX<\infty)=0\).
\item If \(r_2\ge0\) and \(\theta_2\ge1\), then
\(\Pp_{x,y}(\tauY<\infty)=0\).
\item If \(0\le\theta_2<1\), then
\beqlb \label{eq:basic-clock}
 a_2(1-\theta_2)\E_{x,y}\int_0^{\tau_0}X_s^{\kappa_2}\dd s
 =a_2(1-\theta_2)\E_{x,y}\int_0^{\tauY}X_s^{\kappa_2}\dd s \le y^{1-\theta_2}.
\eeqlb
In particular,
\begin{equation}
 \left\{\int_0^\infty X_s^{\kappa_2}\dd s=\infty\right\}
 \subseteq\{\tauY<\infty\}
 \qquad\text{almost surely}.
 \label{eq:exposure-implies-extinction}
\end{equation}
Moreover, for every \(T,L>0\),
\beqnn
 \Pp_{x,y}\left(
 \tauY>T,
 \ \int_0^T X_s^{\kappa_2}\dd s\ge L
 \right)
 \le\frac{y^{1-\theta_2}}{a_2(1-\theta_2)L}.
 \eeqnn
\end{enumerate}
\end{theorem}
 The exposure criterion in
Theorem~\ref{thm:boundary-clock}\textup{(iii)} has the following direct
consequence when \(X\) admits an eventual polynomial lower bound.

\begin{corollary}
\label{cor:lower-exposure}
 Under the assumptions of
Theorem~\ref{thm:boundary-clock}\textup{(iii)}, suppose that there exist
 constants \(c_0>0\) and \(\beta\ge0\), together with an almost
surely finite random time \(T_0\), such that, almost surely, 
\[
 X_t\ge c_0(1+t)^{-\beta},
\qquad t\ge T_0. 
\]
 If \(\beta\kappa_2\le1\), then \(\tauY<\infty\) almost surely. 
\end{corollary}

 The following result complements the preceding pathwise
criteria by giving a sufficient condition, expressed entirely in terms of the
model parameters, for almost-sure extinction of \(Y\). 

\begin{theorem} 
\label{thm:kappa-le-r}
Assume that the model $(X, Y)$ in \eqref{eq:noimm-system} is
source-conservative. Suppose that $ 0\le\theta_2<1,$ $b_{10}>0,$ and
\beqnn
{ r_1,r_2\ge0,
 \qquad b_{11}+b_{12}>0,
 \qquad b_{21}+b_{22}>0,
 \qquad
 \frac{r_2\kappa_2}{r_2+1-\theta_2}< r_1.}
\eeqnn
Then
\beqnn
 \Pp_{x,y}(\tauY<\infty)=1
 \qquad\text{for every }\ x,y>0.
\eeqnn
\end{theorem}

\begin{remark}
 The implication in \eqref{eq:exposure-implies-extinction} is
not reversible without additional assumptions: finite unweighted exposure
does not preclude extinction caused by the Brownian or stable branching
fluctuations.  In the multiplicative regimes below, the integrating factor
associated with the autonomous dynamics of \(Y\) instead yields a weighted
exposure clock that characterizes \(\tauY\). 
\end{remark}
 
\subsection{Multiplicative regimes}

 In the multiplicative regimes, we use the relative-jump
realization of Proposition~\ref{prop:relative-jump}.  For later use, let
\(\widetilde M_1, \widetilde M_2\) be independent compensated Poisson random measures with respective \vthirtyfour{compensators} \(b_{12}\dd t\,\mu_1(\dd v)\) and
\(b_{22}\dd t\,\mu_2(\dd v)\).  Assume that these
measures and \(B_1,B_2\) are mutually independent. The
associated geometric L\'evy processes \(G_1,G_2\) are defined by 
\begin{empheq}[left=\empheqlbrace]{equation}
	\begin{aligned}
		\dd G_{1,t}
		&=G_{1,t-}\left[
		-b_{10}\dd t+\sqrt{2b_{11}}\dd B_1(t)
		+\int_0^\infty v\,\widetilde M_1(\dd t,\dd v)
		\right],
		\qquad G_{1,0}=1,\\
		\dd G_{2,t}
		&=G_{2,t-}\left[
		-b_{20}\dd t+\sqrt{2b_{21}}\dd B_2(t)
		+\int_0^\infty v\,\widetilde M_2(\dd t,\dd v)
		\right],
		\qquad G_{2,0}=1.
	\end{aligned}
	\label{eq:G2}
\end{empheq}	
 
 We first assume that all active autonomous terms in the
\(Y\)-equation are multiplicative.  The resulting integrating-factor identity
gives an exact weighted-clock characterization of \(\tauY\) and yields
sufficient conditions for almost-sure extinction. 
 
\begin{theorem}
\label{thm:FK-rate}
Assume that the model \eqref{eq:noimm-system} is source-conservative, \(0\le\theta_2<1\), and
\beqnn
	b_{20}=0\ \text{or}\ r_{20}=1,
\qquad b_{21}=0\ \text{or}\ r_{21}=2,
\qquad b_{22}=0\ \text{or}\ r_{22}=\alpha_2.
\eeqnn 
 Let \(G_2\) be given by \eqref{eq:G2}. 
\begin{enumerate}[label=\textup{(\roman*)}]
\item For every \(t<\tau_0\),
\begin{equation}
 Y_t=G_{2,t}\left[
 y^{1-\theta_2}-(1-\theta_2)a_2
 \int_0^tG_{2,u}^{\theta_2-1}X_u^{\kappa_2}\dd u\right]^{1/(1-\theta_2)}.
 \label{eq:FK-identity}
\end{equation}
Moreover,
\begin{equation}
 \tauY=\inf\left\{t\ge0:
 \int_0^{t\wedge\tau_0}G_{2,u}^{\theta_2-1}X_u^{\kappa_2}\dd u
 =\frac{y^{1-\theta_2}}{(1-\theta_2)a_2}\right\}.
 \label{eq:FK-hitting}
\end{equation}
In particular,
\begin{equation}
 \{\tauY=\infty\}
 \subseteq\left\{
 \int_0^{\tau_0}G_{2,u}^{\theta_2-1}X_u^{\kappa_2}\dd u
 \le\frac{y^{1-\theta_2}}{(1-\theta_2)a_2}\right\}.
 \label{eq:FK-survival-budget}
\end{equation}
\item  Assume additionally that $b_{10}>0,$ $b_2>0$, and $r_1 \ge 0$. Then
\begin{enumerate}[label=\textup{(\alph*)}]
\item If \(r_1=0\) and $(1-\theta_2)b_2>\kappa_2b_1,$
then $\Pp_{x,y}(\tauY<\infty)=1$ for every $x,y>0$.
\item If {
\((1-\theta_2)b_2=\kappa_2b_1\)} and
\beqnn
 r_{10}=1,
 \qquad b_{11}=0\ \text{or}\ r_{11}=2,
 \qquad b_{12}=0\ \text{or}\ r_{12}=\alpha_1,
\eeqnn 
then $\Pp_{x,y}(\tauY<\infty)=1$ for every $x,y>0$.
\item If \(r_1>0\), then
{ \(\Pp_{x,y}(\tauY<\infty)=1\) for every \(x,y>0\).}
\end{enumerate}
\end{enumerate}
\end{theorem}

\begin{remark}
	Under the nondegeneracy assumptions \vthirtyfour{in~\cite{RenXiongYangZhou2022}}, the conclusions in
	Theorem~\ref{thm:FK-rate}\textup{(ii)(a)} and~\textup{(ii)(c)} can also be
	recovered from \cite[Theorems~1.7 and~1.9]{RenXiongYangZhou2022} through a
	one-way comparison argument.  The critical equality in
	part~\textup{(ii)(b)} is not covered by those results and is established here.
\end{remark}

\begin{remark}
	The condition \(r_1>0\) in
	Theorem~\ref{thm:FK-rate}\textup{(ii)(c)} places the suppressing coordinate
	\(X\) in a subexponential lower-decay regime, whereas
	\(G_2^{\theta_2-1}\) grows exponentially.  Consequently, the weighted exposure
	diverges for every \(\kappa_2>0\), even though the corresponding unweighted
	exposure need not diverge.
\end{remark}
	
The preceding theorem assumes source-conservativeness and
imposes the multiplicative structure only on the autonomous terms in the
\(Y\)-equation.  We now impose the corresponding structure on both
coordinates.  The resulting factorization of \(X\) establishes
source-conservativeness and reduces the extinction problem to a comparison of
the effective decay rates.  Specifically, assume 
\begin{equation}
 \theta_1=1,
 \qquad 0\le\theta_2<1,
 \qquad
 \begin{cases}
 b_{i0}=0\ \text{or}\ r_{i0}=1,\\
 b_{i1}=0\ \text{or}\ r_{i1}=2,\\
 b_{i2}=0\ \text{or}\ r_{i2}=\alpha_i,
 \end{cases}
 \quad i=1,2.
 \label{eq:fully-multiplicative}
\end{equation}

\begin{theorem}
\label{thm:critical-transition}
 Assume \eqref{eq:fully-multiplicative}, and let \(G_1,G_2\) be
given by \eqref{eq:G2}.  Then assertion~\textup{(i)} below holds.  If, in
addition, \(b_1,b_2>0\), assertions~\textup{(ii)--(iv)} hold as well. 
\begin{enumerate}[label=\textup{(\roman*)}]
{
\item The model is source-conservative, and \(X\) cannot reach zero before
\(Y\); in particular, \(\tau_0=\tauY\) almost surely.  For every
\(t<\tauY\),
\beqlb\label{eq:critical-X-factor}
 X_t=xG_{1,t}\exp\left(a_1\int_0^tY_u^{\kappa_1}\dd u\right).
\eeqlb
Moreover, \eqref{eq:FK-identity}, \eqref{eq:FK-hitting} and \eqref{eq:FK-survival-budget} hold.}
\item If $(1-\theta_2)b_2\ge\kappa_2b_1,$ then
{ $\Pp_{x,y}(\tauY<\infty)=1$ for every $x, y > 0.$}
\item If $(1-\theta_2)b_2<\kappa_2b_1,$
then for every fixed $y>0$,
{
\[
 \lim_{x\downarrow0}\Pp_{x,y}(\tauY<\infty)=0.
\]}
If, in addition, \(b_{11}+b_{12}>0\), then
{
\[
 \Pp_{x,y}(\tauY<\infty)>0
 \qquad\text{for every }x,y>0.
\]}
Consequently, for each \(y>0\) and all sufficiently small \(x>0\),
{
\[
 0<\Pp_{x,y}(\tauY<\infty)<1.
\]}
\item { If \((1-\theta_2)b_2<\kappa_2b_1\), then,} on
\(\{\tauY=\infty\}\),
{
\[
 \int_0^\infty X_t^{\kappa_2}\dd t<\infty,
 \qquad
 \lim_{t\to\infty}\frac1t\log X_t=-b_1,
 \qquad
 Y_t\longrightarrow0.
\]}
 On $\{\tau_0^Y = \infty\}$,  if  $\frac{y^{1-\theta_2}}{(1-\theta_2)a_2}- 
 \int_0^\infty G_{2,t}^{\theta_2-1}X_t^{\kappa_2}\dd t>0,$ 
then
\[
 \lim_{t\to\infty}\frac1t\log Y_t=-b_2.
\] 
 On the same event $\{\tau_0^Y = \infty\}$, if $\frac{y^{1-\theta_2}}{(1-\theta_2)a_2}-
 \int_0^\infty G_{2,t}^{\theta_2-1}X_t^{\kappa_2}\dd t=0,$  then
\[
 \limsup_{t\to\infty}\frac1t\log Y_t\le-b_2.
\]
\end{enumerate}
\end{theorem}

\begin{remark}
 By Theorem~\ref{thm:critical-transition}\textup{(i)},
\(\tau_0=\tauY\) almost surely.  Hence, on \(\{\tauY=\infty\}\),
\eqref{eq:FK-survival-budget} gives
\[
 \frac{y^{1-\theta_2}}{(1-\theta_2)a_2}
 -\int_0^\infty
 G_{2,t}^{\theta_2-1}X_t^{\kappa_2}\dd t
 \ge0.
\]
Thus the strict and equality cases in part~\textup{(iv)} exhaust all
possibilities on the event of nonextinction.  If the difference above is negative, then
\eqref{eq:FK-survival-budget} rules out \(\{\tauY=\infty\}\), and hence
\(\tauY<\infty\). 
\end{remark}

\begin{remark}
 Under the effective-rate normalization used here, the strict
inequality \((1-\theta_2)b_2<\kappa_2b_1\) is the counterpart of the threshold
obtained for the one-way model in \cite[Example~1.12\textup{(iv)(a)}]{RenXiongYangZhou2022}.
Part~\textup{(iv)} additionally identifies the joint pathwise decay of the
two coordinates in the present feedback system. 
\end{remark}

 The deterministic fully multiplicative model makes the
weighted-clock comparison explicit.  When the Brownian and stable-noise terms
vanish, \(G_{i,t}=\e^{-b_{i0}t}\), and the weighted clock yields explicit
sufficient conditions for finite-time extinction and nonextinction of \(Y\).

\begin{corollary}
	\label{cor:deterministic-zero-prob}
	 Assume \eqref{eq:fully-multiplicative}, and suppose that 
	$b_{11}=b_{12}=b_{21}=b_{22}=0$ and $b_{10},b_{20}>0.$
	If
	\(\kappa_2b_{10}-(1-\theta_2)b_{20}\le0\), then \(Y\) reaches zero in finite
	time.  If \(\kappa_2b_{10}-(1-\theta_2)b_{20}>0\) and
	\begin{equation}
		\frac{x^{\kappa_2}}{\kappa_2b_{10}-(1-\theta_2)b_{20}}
		\exp\left(\frac{\kappa_2a_1y^{\kappa_1}}{\kappa_1b_{20}}\right)
		<\frac{y^{1-\theta_2}}{(1-\theta_2)a_2},
		\label{eq:det-survival-condition}
	\end{equation}
		then \(Y\) never reaches zero.   On the other hand, if 
	\(\kappa_2b_{10}-(1-\theta_2)b_{20}>0\) and
	\begin{equation}
		\frac{x^{\kappa_2}}{\kappa_2b_{10}-(1-\theta_2)b_{20}}
		>\frac{y^{1-\theta_2}}{(1-\theta_2)a_2},
		\label{eq:det-extinction-condition}
	\end{equation}
	then \(Y\) reaches zero in finite time.  
\end{corollary}

\begin{remark}
 When
\(\kappa_2b_{10}-(1-\theta_2)b_{20}>0\), \vthirtyfour{Corollary~\ref{cor:deterministic-zero-prob}} exhibits both
finite-time extinction and nonextinction, depending on the initial state.
The intermediate region not covered by
\eqref{eq:det-survival-condition} or
\eqref{eq:det-extinction-condition} is not classified by this result. 
\end{remark}

 The remainder of the paper is organized as follows.
Section~\ref{sec:preliminaries} develops the deterministic classification,
the lower comparison process, and the relative-jump representation.
Section~\ref{sec:proofs} contains the proofs of the main results.  The appendix
collects the well-posedness and comparison results together with the stable
integral identity. 

\section[\vthirtyfour{Preliminary results}]{Preliminary results}
\label{sec:preliminaries}
\label{sec:structural}

\subsection{The deterministic interaction skeleton}

\vthirtyfour{Let \((x(t),y(t))_{0\le t<T_*}\) be the positive solution of}
\begin{equation}
	\dot x=a_1x^{\theta_1}y^{\kappa_1},
	\qquad
	\dot y=-a_2y^{\theta_2}x^{\kappa_2}.
	\label{eq:interaction-ODE}
\end{equation} 
from \(x_0,y_0>0\), where
\[
 T_*:=\sup\{T>0:x(t)>0\text{ and }y(t)>0\text{ for }0\le t<T\}.
\]
Let $h_0:=H(x_0,y_0)$, where $H$ is given by \eqref{eq:H}. In fact, \(H\) is a first integral of the dynamic system \eqref{eq:interaction-ODE}.
The following proposition gives the complete classification of the
deterministic interaction system. Although it is not used in the proofs of the
stochastic extinction results, it provides a natural view by isolating the
boundary and long-time behavior generated by the interaction alone, against
which the effects of the Brownian and stable branching noises can be assessed.

\begin{proposition} 
\label{prop:ODE-classification}
  One has \(H(x(t),y(t))=h_0\) for \(0\le t<T_*\), and  one of the
following cases holds:
\begin{enumerate}[label=\textup{(\roman*)}]
\item If \(q_2>0\) and  
\begin{itemize}
	\item[(a)] \(q_1\ge0\) or
	\item[(b)]\(q_1<0\) and \(h_0<0\),
\end{itemize}
 then
\[
 y(t)\downarrow0,
 \qquad x(t)\uparrow x_\infty\in(x_0,\infty),
 \qquad \Phi_{q_1}(x_\infty)=h_0.
\]
Moreover, \(T_*<\infty\) if and only if \(\theta_2<1\).
\item If \(q_1<0\) and
\begin{itemize}
	\item[(a)] \(q_2\le0\)  or
	\item[(b)] \(q_2>0\) and \(h_0>0\),
\end{itemize} 
there is a unique \(y_*\in(0,y_0)\) such that
\[
 h_0-\frac{a_1}{a_2}\Phi_{q_2}(y_*)=0,
\]
and
\[
 x(t)\uparrow\infty,
 \qquad y(t)\downarrow y_*>0,
 \qquad T_*<\infty.
\]
\item If \(q_1<0<q_2\) and \(h_0=0\), then \(x(t)\uparrow\infty\) and
\(y(t)\downarrow0\).  {Moreover, \(T_*<\infty\) if and only if}
\[
 \kappa_1\kappa_2-(\theta_1-1)(\theta_2-1)>0.
\]
\item If \(q_1>0>q_2\), then \(x(t)\uparrow\infty\) and
\(y(t)\downarrow0\).  {Moreover, \(T_*<\infty\) if and only if}
\[
 \kappa_1\kappa_2-(\theta_1-1)(\theta_2-1)<0.
\]
\item If \(q_1=0>q_2\), then \(x(t)\uparrow\infty\),
\(y(t)\downarrow0\) and \(T_*<\infty\).
\item {Suppose \(q_2=0\) and \(q_1\ge0\).  If
\(q_1>0\), then \(x(t)\uparrow\infty\), \(y(t)\downarrow0\) and
\(T_*=\infty\).}  If
\(q_1=q_2=0\), then {\(x(t)\uparrow\infty\),
\(y(t)\downarrow0\) and}
\[
 x(t)y(t)^{a_1/a_2}=x_0y_0^{a_1/a_2},
\]
and \(T_*<\infty\) if and only if $\frac{a_1}{a_2}\kappa_2>\kappa_1.$
\end{enumerate}
\end{proposition}

\begin{proof} 
{ For $x, y > 0$, we have $\dot x=a_1x^{\theta_1}y^{\kappa_1}>0$ and $\dot y=-a_2 y^{\theta_2}x^{\kappa_2}<0$ by \eqref{eq:interaction-ODE}.
Thus the function \(t\mapsto x(t)\) is strictly increasing and  \(t\mapsto y(t)\) is strictly decreasing on $[0, T_*)$.}  By \eqref{eq:interaction-cancel},
\begin{equation}
 {H(x(t),y(t))=}
 \Phi_{q_1}(x(t))+\frac{a_1}{a_2}\Phi_{q_2}(y(t))=h_0.
 \label{eq:levelset}
\end{equation}
Since \(y\) is strictly monotone, the map
$t\mapsto y(t)$ has an inverse on its
range. We may therefore reparametrize the trajectory by $y$, writing
$x(y):=x(t(y))$.  Define $ R(y):=h_0-\frac{a_1}{a_2}\Phi_{q_2}(y).$
Then
\beqlb\label{eq0821a}
 \Phi_{q_1}(x(y))=R(y),
 \qquad
 R'(y)=-\frac{a_1}{a_2}y^{q_2-1}<0,
 \qquad
 R(y_0)=\Phi_{q_1}(x_0).
\eeqlb
The ranges of \(\Phi_{q_1}\) are
\beqnn 
 \Phi_{q_1}((0,\infty))=
 \begin{cases}
  (0,\infty),&q_1>0,\\
  \R,&q_1=0,\\
  (-\infty,0),&q_1<0,
 \end{cases}
\quad
\text{and}
\quad 
 \lim_{y\downarrow0}\Phi_{q_2}(y)=
 \begin{cases}
  0,&q_2>0,\\
  -\infty,&q_2\le0.
 \end{cases}
\eeqnn 
Because \(R\) increases when \(y\) decreases, these range identities give all possible endpoints.

{
\par\smallskip
\noindent\textup{(i)}\quad
If \(q_2>0\), then \(R(y)\uparrow h_0\) as \(y\downarrow0\).  When
\(q_1>0\), by \eqref{eq:levelset}, one sees that
\(h_0>0\) and hence \(h_0\in\Phi_{q_1}((0,\infty))=(0,\infty)\).
When \(q_1=0\), the range of \(\Phi_0 \) is \(\R\).  Finally, when
\(q_1<0\) and \(h_0<0\), one has
\(h_0\in\Phi_{q_1}((0,\infty))=(-\infty,0)\).  Therefore, in all the
cases covered by \textup{(i)}, we have $
x(y)\rightarrow x_\infty:=\Phi_{q_1}^{-1}(h_0)\in(0,\infty).$
Moreover, since \(q_2>0\),
\(h_0=\Phi_{q_1}(x_0)+(a_1/a_2)\Phi_{q_2}(y_0)
>\Phi_{q_1}(x_0)\), so \(x_\infty>x_0\).
 
\par\smallskip
\noindent\textup{(ii)}\quad
Suppose that \(q_1<0\).  If \(q_2\le0\), then
\(R(y)\to+\infty\) as \(y\downarrow0\); if \(q_2>0\) and \(h_0>0\),
then \(R(y)\to h_0>0\).  Since
\(R(y_0)=\Phi_{q_1}(x_0)<0\), there exists a unique \(y_*\in(0,y_0)\) such that
\(R(y_*)=0\) \vthirtyfour{when \(q_2\le0\), or when \(q_2>0\) and \(h_0>0\). Moreover,} as \(y\downarrow y_*\), by \eqref{eq0821a},
\[
 R(y)<0,
 \qquad
 x(y)=\bigl(q_1R(y)\bigr)^{1/q_1}\uparrow\infty.
\]
}
\par\smallskip
\noindent\textup{(iii)}\quad
If \(q_1<0<q_2\) and \(h_0=0\), then \(R(y)<0\) for every \(y>0\) and \(R(y)\uparrow0\) only as \(y\downarrow0\); hence \(x(y)\uparrow\infty\) as \(y\downarrow0\).

\par\smallskip
\noindent\textup{(iv)}\quad
If \(q_1>0>q_2\), then \(R(y)\to+\infty\) and therefore \(x(y)=\bigl(q_1R(y)\bigr)^{1/q_1}\uparrow\infty\) as \(y\downarrow0\).

\par\smallskip
\noindent\textup{(v)}\quad
{The case \(q_1=0>q_2\) follows from the same range calculation, with \(\Phi_0(x)=\log x\).}

\par\smallskip
\noindent\textup{(vi)}\quad
The case \(q_2=0\) and \(q_1\ge0\) follows from the same range calculation, with \(\Phi_0(y)=\log y\). 

It remains to determine whether the endpoint is reached in finite time.  Since
\[
 \frac{\dd t}{\dd y}=-\frac1{a_2 y^{\theta_2}x(y)^{\kappa_2}},
\]
and \(y(t)\) decreases from \(y_0\) to its endpoint
\(y_*\), a change of variables gives the maximal travel time 
\begin{equation}
  T_*
 =-\frac1{a_2}\int_{y_0}^{y_*}y^{-\theta_2}x(y)^{-\kappa_2}\dd y
 =\frac1{a_2}\int_{y_*}^{y_0}y^{-\theta_2}x(y)^{-\kappa_2}\dd y,
 \label{eq:time-integral}
\end{equation}  
where \(y_*=0\) unless the level curve ends at a positive value of \(y\).

In case \textup{(i)}, \(x(y)\to x_\infty\in(0,\infty)\), so the integrand in \eqref{eq:time-integral} is comparable with \(y^{-\theta_2}\).  Hence \(T_*<\infty\) if and only if \(\theta_2<1\).

 In case \textup{(ii)}, \(y_*>0\) and 
\(R'(y_*)=-(a_1/a_2)y_*^{q_2-1}<0\).  Consequently,
\[
 -R(y)=\frac{a_1}{a_2}y_*^{q_2-1}(y-y_*)\bigl(1+o(1)\bigr),
 \qquad y\downarrow y_*.
\]
Since \(q_1<0\), this implies
\beqnn
 x(y)^{-\kappa_2}
  =\bigl(q_1R(y)\bigr)^{-\kappa_2/q_1} =\left(\frac{(-q_1)a_1}{a_2}y_*^{q_2-1}\right)^{\kappa_2/|q_1|}
   (y-y_*)^{\kappa_2/|q_1|}\bigl(1+o(1)\bigr).
 \eeqnn
The factor \(y^{-\theta_2}\) is bounded near \(y_*>0\), while
\(\kappa_2/|q_1|>0\); hence the integrand in
\eqref{eq:time-integral} is integrable at \(y_*\), and \(T_*<\infty\).

For \textup{(iii)}, if \(q_1<0<q_2\) and \(h_0=0\), then \eqref{eq:levelset} gives
\[
 \frac{x(y)^{q_1}}{q_1}+\frac{a_1}{a_2}\frac{y^{q_2}}{q_2}=0 
 \qquad \text{and} \qquad
 x(y)=Cy^{q_2/q_1}
\]
for a constant \(C>0\).  The integrand in \eqref{eq:time-integral}
is a constant multiple of
\(y^{-\theta_2-\kappa_2q_2/q_1}\), and
\[
 q_1\left(\theta_2+\frac{\kappa_2q_2}{q_1}-1\right)
 =\kappa_1\kappa_2-(\theta_1-1)(\theta_2-1).
\]
Since \(\int_0^1y^{-a}\dd y<\infty\) if and only if
\(a<1\), and \(q_1<0\), it follows that
\beqnn
 T_*<\infty
 \quad\Longleftrightarrow\quad
 \kappa_1\kappa_2-(\theta_1-1)(\theta_2-1)>0.
\eeqnn

For \textup{(iv)}, if \(q_1>0>q_2\), then \eqref{eq0821a} gives
{
\[
 R(y)=h_0-\frac{a_1}{a_2q_2}y^{q_2}
 \sim-\frac{a_1}{a_2q_2}y^{q_2},
 \qquad y\downarrow0,
\]
because \(q_2<0\).  Hence \(x(y)\asymp y^{q_2/q_1}\), and the same
power comparison as above applies.  Since \(q_1>0\),
\[
 T_*<\infty
 \quad\Longleftrightarrow\quad
 \kappa_1\kappa_2-(\theta_1-1)(\theta_2-1)<0.
\]}

For \textup{(v)}, if \(q_1=0>q_2\), then
\[
 x(y)=\exp\left(h_0-\frac{a_1}{a_2q_2}y^{q_2}\right).
\]
{In particular, \(x(y)\uparrow\infty\) as
\(y\downarrow0\), and
\[
 y^{-\theta_2}x(y)^{-\kappa_2}
 =y^{-\theta_2}
  \exp\left(-\kappa_2h_0
  +\frac{a_1\kappa_2}{a_2q_2}y^{q_2}\right).
\]
Since \(q_2<0\), the exponential factor tends to zero faster than any
positive power of \(y\) as \(y\downarrow0\).  Thus the last expression is
integrable at zero, and \(T_*<\infty\).}

For \textup{(vi)}, let \(q_2=0\), so \(\theta_2=1+\kappa_1\).
{The case \(q_1<0\) has already been covered by
\textup{(ii)}.  If \(q_1>0\), then, as \(y\downarrow0\),}
\[
 x(y)^{q_1}=q_1\bigl(h_0-(a_1/a_2)\log y\bigr)\asymp|\log y|,
\]
{and therefore, as \(y\downarrow0\),}
\[
 y^{-\theta_2}x(y)^{-\kappa_2}
 \asymp y^{-1-\kappa_1}|\log y|^{-\kappa_2/q_1},
\]
{which is not integrable at zero because
\(\kappa_1>0\).  Hence \(T_*=\infty\).}  If \(q_1=q_2=0\), then
\[
 \log x(y)+\frac{a_1}{a_2}\log y=h_0,
 \qquad
 x(y)=\e^{h_0}y^{-a_1/a_2},
\]
{so \(x(y)\uparrow\infty\) as \(y\downarrow0\).  Moreover,}
\[
 T_*<\infty
 \Longleftrightarrow
 \int_0^\varepsilon y^{-\theta_2+(a_1/a_2)\kappa_2}\dd y<\infty
 \Longleftrightarrow
 \frac{a_1}{a_2}\kappa_2>\theta_2-1=\kappa_1.
\]
This proves all cases.
\end{proof}

\subsection[The lower comparison process]{\vnew{The lower comparison process}}

Now we consider the following process driven by the same noises as $X$:  
\beqlb\label{eq:lower-envelope}
 \Xlow_t= x-b_{10}\int_0^t\Xlow_s^{r_{10}}\dd s
 +\int_0^t\sqrt{2b_{11}\Xlow_s^{r_{11}}}\dd B_1(s)
 +\int_0^t\int_0^\infty\int_0^{b_{12}\Xlow_{s-}^{r_{12}}}
 z\,\widetilde N_1(\dd s,\dd z,\dd u).
\eeqlb

  The process \(\Xlow\) defined by \eqref{eq:lower-envelope} is a
	continuous-state nonlinear branching process studied systematically
	\vthirtyfour{in~\cite{LiYangZhou2019}}; see also \vthirtyfour{Li~\cite{Li2019Polynomial}}.  \vthirtyfour{These works provide a general analysis} of its boundary and long-time behavior.  The following proposition
	refines these results in the present setting by identifying the almost-sure
	logarithmic decay rate of \(\Xlow\).  

\begin{proposition} 
\label{prop:lower-log-rate}
{ Assume $b_{10}>0$ and $r_1 \ge 0$.} Then the solution of \eqref{eq:lower-envelope} is global and strictly positive. Moreover, $\Xlow_t\rightarrow0$ as $t \rightarrow \infty$, and
\beqnn
 \lim_{t\to\infty}\frac1t\log\Xlow_t
 =-b_1\1_{\{r_1=0\}}
 \quad\text{almost surely}.
 \eeqnn
\end{proposition}

{
\begin{proof}
	If $b_{11}+b_{12}=0$, the result follows directly from the corresponding
	ordinary differential equation. Hence we assume $b_{11}+b_{12}>0$. The global existence, strict
	positivity, and convergence \(\Xlow_t\to0\) follow from
	\cite[Theorem~3.1 and Example~2.18]{LiYangZhou2019}. It remains to prove the logarithmic limit. 
	
	For the stable L\'evy measure $\mu_1$, we have
	$\int_0^\infty[\log(1+v)]^2\mu_1(\dd v)<\infty,$
	and by Lemma \ref{lem:stable-power} with $q=0$, we have
	\beqnn
		\int_0^\infty
		\bigl[v-\log(1+v)\bigr]\mu_1(\dd v)=1.
		\eeqnn
	Applying It\^o's formula with jumps to \(\log\Xlow_t\), first between
	interior localization times and then removing the localization using
	global strict positivity, yields
	\begin{align}
		\log\Xlow_t
		={}&\log x-
		\int_0^t\left[
		b_{10}\Xlow_s^{r_{10}-1}
		+b_{11}\Xlow_s^{r_{11}-2}
		+b_{12}\Xlow_s^{r_{12}-\alpha_1}
		\right]\dd s
		+M_t^B+M_t^J,
		\label{eq:log-lower}
	\end{align}
	where
	\begin{align*}
		M_t^B
		:={}&
		\int_0^t
		\sqrt{2b_{11}\Xlow_s^{r_{11}-2}}\,\dd B_1(s),\\
		M_t^J
		:={}&
		\int_0^t\int_0^\infty
		\int_0^{b_{12}\Xlow_{s-}^{r_{12}}}
		\log\left(1+\frac{z}{\Xlow_{s-}}\right)
		\widetilde N_1(\dd s,\dd z,\dd u)
	\end{align*}
	are {local martingales.}
Since \(r_1\ge0\), every active exponent among
\(r_{10}-1\), \(r_{11}-2\), and \(r_{12}-\alpha_1\) is
nonnegative.
The predictable quadratic variation of $M^B$ is $\langle M^B\rangle_t=2b_{11}\int_0^t\Xlow_s^{r_{11}-2}\dd s.$ If \(b_{11}=0\), this bracket vanishes; otherwise
\(r_{11}-2\ge0\), and \(\Xlow_t\to0\) shows that its integrand is
eventually bounded.  Consequently, $\int_0^\infty
\frac{\dd\langle M^B\rangle_s}{(1+s)^2} < \infty$ almost surely. Therefore, $\int_0^t\frac{1}{1+s}\,\dd M_s^B$
	converges almost surely as $t\to\infty$. Similarly, for the jump term, the predictable quadratic variation is
	\begin{align*}
		\langle M^J\rangle_t
		={}&b_{12}\int_0^t\Xlow_s^{r_{12}}
		\int_0^\infty
		\left[
		\log\left(1+\frac{z}{\Xlow_s}\right)
		\right]^2
		\mu_1(\dd z)\dd s\\
		={}&b_{12}
		\left(
		\int_0^\infty[\log(1+v)]^2\mu_1(\dd v)
		\right)
		\int_0^t\Xlow_s^{r_{12}-\alpha_1}\dd s.
	\end{align*}
 If \(b_{12}=0\), this bracket vanishes;
	otherwise \(r_{12}-\alpha_1\ge0\), and its integrand is eventually
	bounded.  Thus, in either case, $\int_0^\infty
	\frac{\dd\langle M^J\rangle_s}{(1+s)^2}
	<\infty$ almost surely. Hence $\int_0^t\frac{1}{1+s}\,\dd M_s^J$
	also converges almost surely; see, e.g., \vthirtyfour{Karatzas and Kardaras}~\cite[Remark~7.2]{KaratzasKardaras2007}. 
	
	For either \(L=M^B\) or \(L=M^J\), by
	integration by parts, one obtains
	\beqnn 
	\frac{L_t}{1+t}
	=
	\int_0^t\frac{1}{1+s}\,\dd L_s
	-
	\frac{1}{1+t}
	\int_0^t
	\left(
	\int_0^s\frac{1}{1+u}\,\dd L_u
	\right)\dd s.
	\eeqnn
 Since the inner stochastic integral
	converges, \vthirtyfour{Ces\`aro convergence (see, e.g., Hardy~\cite[Chapter~III]{Hardy1949}) yields \(M_t^B/t\to0\) and \(M_t^J/t\to0\) almost surely as $t \rightarrow \infty$.}
 \vthirtyfour{These limits also follow from the continuous-time
	Kronecker lemma of Elliott}~\cite[Theorem~2.2]{Elliott2001}.
 Finally, for every \(a>0\), the convergence
	\(\Xlow_t\to0\) implies
	\(t^{-1}\int_0^t\Xlow_s^a\dd s\to0\) almost surely.  Hence the
	terms in \eqref{eq:log-lower} with positive effective exponent vanish
	after division by \(t\), whereas those with exponent zero contribute
	the sum of their coefficients, namely \(b_1\).  Together with the
	martingale limits above, this gives 
	\[
	\lim_{t\to\infty}\frac1t\log\Xlow_t
	=
	-b_1\1_{\{r_1=0\}}
	\qquad\text{almost surely}.
	\]
	This completes the proof.
\end{proof}
}

\subsection[Relative-jump representation and geometric L\'evy factors]{\vnew{Relative-jump representation and geometric L\'evy factors}}
This subsection develops the relative-jump representation and the geometric
L\'evy factors used in the multiplicative regimes.

The following proposition establishes the relative-jump representation used in
the multiplicative regimes.  At the level of the joint weak law, each
multiplicative stable-jump term is represented using an independent compensated
Poisson random measure with jump sizes proportional to the current state,
while the corresponding stopped law is preserved.  This weak-law construction,
which does not involve a state-dependent transformation of the original Poisson
random measures, underlies the integrating-factor formula for \(Y\) in
Theorem~\ref{thm:FK-rate} and the factorization of \(X\) in
Theorem~\ref{thm:critical-transition}.

  \begin{proposition}
		\label{prop:relative-jump}
		Let \(I\subseteq\{1,2\}\), and assume that \(r_{i2}=\alpha_i\) whenever
		\(i\in I\) and \(b_{i2}>0\).  Then the solution to
		\rev{\eqref{eq:noimm-system}} can be realized weakly up to
		\(\tau_0\wedge\tau_\infty\) so that all driving Brownian motions and Poisson
		random measures are mutually independent and the stable-jump terms have the
		following representations:
		\begin{enumerate}[label=\textup{(\roman*)}]
			\item  for each \(i\notin I\), the \(i\)-th stable-jump term is represented
			in the original  Poisson-integral form of
			\rev{\eqref{eq:noimm-system}};
			\item for each \(i\in I\), there exists a Poisson random measure \(M_i\) on
			\(\Rplus\times(0,\infty)\), with compensator
			\(b_{i2}\dd t\,\mu_i(\dd v)\), such that the \(i\)-th stable-jump term is
			\[
			\begin{cases}
					X_{t-}\displaystyle\int_0^\infty \vthirtyfour{v}\,\widetilde M_1(\dd t,\dd v),
				& i=1,\\[1ex]
					Y_{t-}\displaystyle\int_0^\infty \vthirtyfour{v}\,\widetilde M_2(\dd t,\dd v),
				& i=2.
			\end{cases}
			\]
		\end{enumerate}
		If the model is source-conservative, the source-stopped law of this
		realization coincides with that of the original formulation.
\end{proposition}

\begin{proof}
	Let \(Z=(X,Y)\) be the solution to
	\eqref{eq:noimm-system}.  On a filtered probability space carrying the
	mutually independent Brownian motions and Poisson random measures \vthirtyfour{specified above}, consider the system
	\(\widehat Z=(\widehat X,\widehat Y)\) obtained by replacing, for each
	\(i\in I\), the \(i\)-th stable-jump term with
	\[
	\int_0^t\int_0^\infty
	\widehat Z_{i,s-}v\,\widetilde M_i(\dd s,\dd v),
	\qquad
	\widehat Z_{1}=\widehat X,\quad \widehat Z_{2}=\widehat Y,
	\]
	and leaving all other terms unchanged.  For \(z,z'>0\),
	\[
	b_{i2}\int_0^1|zv-z'v|^2\mu_i(\dd v)
	=
	b_{i2}|z-z'|^2\int_0^1v^2\mu_i(\dd v)<\infty,
	\]
	and $\int_1^\infty(1+v)\mu_i(\dd v)<\infty.$
	Thus the relative small-jump coefficient is Lipschitz in the corresponding
	\(L^2(\mu_i)\)-norm, while the large-jump component has finite activity and
	finite first moment.  Since the remaining coefficients are unchanged, the
	localization construction used in the proof of
	Lemma~\ref{prop:local-wellposedness} yields a weak c\`adl\`ag solution
	\(\widehat Z\) up to its first boundary or explosion time.
	
We identify the law of \(\widehat Z\) through  \(\cL\).
Let \(F\in C^2((0,\infty)^2)\) be nonnegative and bounded, with all its
first- and second-order partial derivatives bounded.  Write
\(\mathbf z=(z_1,z_2)\) and let \(e_i\) denote the \(i\)-th coordinate vector.
For \(i\in I\), \vthirtyfour{set \(u=z_iv\). The relative-jump term contributes}
\beqnn 
\ar\ar b_{i2}\int_0^\infty
\Bigl[
F(\mathbf z+z_iv e_i)-F(\mathbf z)
-z_iv\,\partial_iF(\mathbf z)
\Bigr]\mu_i(\dd v)\cr 
\ar\ar\qquad =
b_{i2}z_i^{\alpha_i}\int_0^\infty
\Bigl[
F(\mathbf z+u e_i)-F(\mathbf z)
-u\,\partial_iF(\mathbf z)
\Bigr]\mu_i(\dd u) 
\eeqnn
to \(\cL F(\mathbf z)\). 
If \(b_{i2}>0\), then \(r_{i2}=\alpha_i\), and the right-hand side is
precisely the \(i\)-th stable-jump term in \(\cL F\) for the original
formulation.  
All remaining terms in \(\cL F\) agree by construction.  Then \(Z\) and \(\widehat Z\) solve the same local martingale problem
on \((0,\infty)^2\).
	
	For \(n\ge2\), let \(D_n=(n^{-1},n)^2\), and denote by \(\sigma_n\) and
	\(\widehat\sigma_n\) the first exit times of \(Z\) and \(\widehat Z\),
	respectively, from \(D_n\).  By
	Lemma~\ref{prop:local-wellposedness}, the martingale problem for
	\(\cL\) stopped upon exiting \(D_n\) is well posed.  Consequently, the laws
	of \(Z_{\cdot\wedge\sigma_n}\) and
	\(\widehat Z_{\cdot\wedge\widehat\sigma_n}\) coincide for every sufficiently
		large \(n\).  \vthirtyfour{Letting \(n\to\infty\) and applying the localization
		theorem for martingale problems of Ethier and Kurtz}~
		\cite[Chapter~4, Sections~4.4 and~4.6]{EthierKurtz1986}, \vthirtyfour{we conclude that the two}
	laws agree up to \(\tau_0\wedge\tau_\infty\).  The asserted independence and
	jump representations hold by construction.
	
Finally, if the original model is source-conservative, equality of the
interior laws implies that \(\widehat Z\) is source-conservative as well.
Since stopping and freezing a path at its first boundary time is a measurable
transformation, the corresponding source-stopped laws coincide.
\end{proof}

  The following result
 	establishes the strict positivity, logarithmic decay rates, and
 	power-integrability of the associated geometric L\'evy factors.  These
 	properties are used to analyze the weighted exposure and the critical case in
 	Theorem~\ref{thm:FK-rate}, and to prove the effective-rate transition,
 	positive extinction probability, and pathwise decay results in
 	Theorem~\ref{thm:critical-transition}. 

\begin{lemma}
\label{lem:geometric-factor}
For \(i=1,2\), let \(G_i\) be the geometric factor defined in
\eqref{eq:G2}.  Then
\(G_{i,t}>0\) for every \(t\ge0\), and
\begin{equation}
 \log G_{i,t}
 =-b_it+\sqrt{2b_{i1}}B_i(t)
 +\int_0^t\int_0^\infty
 \log(1+v)\,\widetilde M_i(\dd s,\dd v),
 \label{eq:geometric-log}
\end{equation}
where \(b_i:=b_{i0}+b_{i1}+b_{i2}\). Consequently, \(\E|\log G_{i,t}|<\infty\) and
\[
 \lim_{t\to\infty}\frac1t\log G_{i,t}=-b_i
 \qquad\text{almost surely}.
\]
If \(b_i>0\), then, for every \(q>0\),
\(\int_0^\infty G_{i,t}^q\dd t<\infty\) almost surely. 
\end{lemma}

\begin{proof}
For each \(i\), equation~\eqref{eq:G2} is a stochastic exponential whose
relative jumps are \(v>0\); hence \(G_{i,t}>0\) for all \(t\ge0\).
\vthirtyfour{It\^o's formula for \(\log G_{i,t}\) gives}
\begin{align*}
 \log G_{i,t}={}&-b_{i0}t-b_{i1}t
 +\sqrt{2b_{i1}}B_i(t)
 +\int_0^t\int_0^\infty
 \log(1+v)\,\widetilde M_i(\dd s,\dd v)\\
 &+b_{i2}t\int_0^\infty
 \bigl[\log(1+v)-v\bigr]\mu_i(\dd v).
\end{align*}
By Lemma \ref{lem:stable-power}, the last integral equals \(-1\), which proves
\eqref{eq:geometric-log}.  Moreover, recall that
\(\int_0^\infty[\log(1+v)]^2\mu_i(\dd v)<\infty\), so the stochastic integral
term in \eqref{eq:geometric-log} is a square-integrable L\'evy
martingale.  In particular, \(\E|\log G_{i,t}|<\infty\). \vthirtyfour{By the strong law
for L\'evy processes of Sato}~\cite[Theorem~36.5]{Sato1999}, \vthirtyfour{we have}
\[
 \frac1t\left[
 \sqrt{2b_{i1}}B_i(t)
 +\int_0^t\int_0^\infty
 \log(1+v)\,\widetilde M_i(\dd s,\dd v)
 \right]\longrightarrow0
 \quad\text{almost surely}.
\]
\vthirtyfour{Dividing \eqref{eq:geometric-log} by \(t\) yields \(t^{-1}\log G_{i,t}\to-b_i\) almost surely as \(t\to\infty\).}
If \(b_i>0\), then, for every \(\varepsilon\in(0,b_i)\), almost surely
there is a finite random \(T_\varepsilon\) such that
\(G_{i,t}\le\exp\{-(b_i-\varepsilon)t\}\) for all sufficiently large \(t\). Consequently, \vthirtyfour{\(\int_0^\infty G_{i,t}^q\dd t<\infty\) almost surely for \(q>0\).} The result follows. 
\end{proof}
 
 \begin{lemma} 
 	\label{lem:clock-unbounded-support}
 	{Assume \(b_{11}+b_{12}>0\) {and $0 < \theta_2 < 1$.} Then, for every
 		\(T,K>0\), we have
 		\[
 		\Pp\left(
 		\int_0^T G_{2,t}^{\theta_2-1}G_{1,t}^{\kappa_2}\dd t>K
 		\right)>0.
 		\]}
 \end{lemma}
 
 \begin{proof}
 	{ By Proposition~\ref{prop:relative-jump}, the processes
 		\(\log G_1\) and \(\log G_2\) are independent L\'evy processes.}
 	We first prove that, for every \(a,T_0>0\),
 	\begin{equation}
 		\Pp\left(\sup_{0\le t\le T_0}\log G_{1,t}\ge a\right)>0.
 		\label{eq:H1-high}
 	\end{equation}
 	If \(b_{11}>0\), then \(\log G_{1,t_0}\) has a nondegenerate Gaussian
 	component for every \(t_0>0\), and therefore \eqref{eq:H1-high} follows. 
 	
 	Suppose \(b_{11}=0<b_{12}\). By \eqref{eq:geometric-log}, let $\log G_{1,t}=R_t+J_t,$
 	where
 	\beqnn
	J_t:=\int_0^t\int_{\{\log(1+\vthirtyfour{v})>1\}}\log(1+v)M_1(\dd s,\dd v) 
 	\eeqnn 
 	and
 	\beqnn
 	R_t :=-b_1 t
 	-b_{12}t\int_{\{\log(1+v)>1\}}\log(1+v)\mu_1(\dd v) +\int_0^t\int_{\{\log(1+v)\le1\}}\log(1+v)
 	\widetilde M_1(\dd s,\dd v).
 	\eeqnn 
 	The process \(J\) is a compound Poisson process with positive jumps and is independent of the residual L\'evy process \(R\).  Since \(\inf_{0\le t\le T_0}R_t> -\infty\) almost surely, choose \(K<\infty\) such that
 	\[
 	\Pp\left(\inf_{0\le t\le T_0}R_t\ge-K\right)>0.
 	\]
 	The L\'evy measure of \(J\) has unbounded support, so the event that \(J\) has a jump larger than \(a+K\) before \(T_0\) has positive probability.  On its intersection with the preceding residual event,
 	\(\sup_{t\le T_0}\log G_{1,t}\ge a\).  This proves \eqref{eq:H1-high}.
 	
 	{
 		Fix \(T>0\).  Since \(\log G_{1,0}=0\),  \(\log G_1\) is  
		\vthirtyfour{right-continuous at zero}. One may choose
 		\(\delta\in(0,T/4)\) such that 
 		\[
 		\Pp\left(\inf_{0\le u\le\delta}\log G_{1,u}\ge-1\right)>0.
 		\]
 		For \(a>0\), let $\sigma_a^+:=\inf\{t\ge0:\log G_{1,t}\ge a\}.$
 		By \eqref{eq:H1-high}, \(\Pp(\sigma_a^+\le T/2)>0\).
 		On \(\{\sigma_a^+<\infty\}\), by the strong Markov property, one obtains
 		\beqnn 
 		\Pp\left(
 		\inf_{0\le u\le\delta}
 		\bigl[\log G_{1,\sigma_a^++u}-\log G_{1,\sigma_a^+}\bigr]
 		\ge-1
 		\ \middle|\ \mathcal F_{\sigma_a^+}
 		\right) =
 		\Pp\left(\inf_{0\le u\le\delta}\log G_{1,u}\ge-1\right)>0.
 		\eeqnn 
 		Since \(\sup_{0\le t\le T}\log G_{2,t}<\infty\) almost surely, there exists
 		\(C<\infty\), independently of \(a\), such that
 		\[
 		\Pp\left(\sup_{0\le t\le T}\log G_{2,t}\le C\right)>0.
 		\]
		Recall that \(G_1\) and \(G_2\) are independent. \vthirtyfour{The preceding estimates yield}
 		\beqnn
 		\Pp\Bigl( \sigma_a^+\le T/2,\ 
 		\inf_{0\le u\le\delta}
 		\bigl[\log G_{1,\sigma_a^++u}-\log G_{1,\sigma_a^+}\bigr]\ge-1, 
 		\sup_{0\le t\le T}\log G_{2,t}\le C\Bigr)>0.
 		\eeqnn
 		On this event, for every \(0\le u\le\delta\), we have $\log G_{1,\sigma_a^++u}\ge a-1,$ $\log G_{2,\sigma_a^++u}\le C,$
 		and \(\sigma_a^++\delta<T\).  Consequently,
 		\[
 		\int_0^T G_{2,t}^{\theta_2-1}G_{1,t}^{\kappa_2}\dd t
 		\ge\delta\exp\bigl(\kappa_2(a-1)-(1-\theta_2)C\bigr)
 		\]
		with positive probability.  For any fixed \(K>0\), one may choose \(a\) large enough such that \vthirtyfour{\(\delta\exp\!\bigl(\kappa_2(a-1)-(1-\theta_2)C\bigr)\ge K\)}. The result then follows.}
 \end{proof}

\section[Proofs of the main results]{\vnew{Proofs of the main results}}
\label{sec:proofs}
\label{sec:multiplicative-proofs}

\subsection[Proof of Theorem~\ref{thm:concave-cancellation}]{\vnew{Proof of Theorem~\ref{thm:concave-cancellation}}}

\leavevmode
{
The proof of Theorem~\ref{thm:concave-cancellation} is organized through
the next three lemmas.  Lemma~\ref{lem:concave-Lyapunov} establishes
source-conservativeness, the energy estimate, and the supermartingale
property of \(H\).  Lemma~\ref{lem:conditional-increment} gives a conditional
short-time increment estimate on bounded regions, which
Lemma~\ref{lem:selfreg-vanishing} combines with the energy estimate to show
that both coordinates converge to zero on the event of non-absorption.
}


\begin{lemma} 
\label{lem:concave-Lyapunov}
{
Assume that $0 < q_1, q_2 \le 1$. Then the model \eqref{eq:noimm-system} is source-conservative. Moreover,
\eqref{eq:concave-energy} holds, and
\(\bigl(H(X_t,Y_t)\bigr)_{t\ge0}\) is a nonnegative supermartingale.
}
\end{lemma}

\begin{proof}
We extend \(\Phi_{q_1}\) and \(\Phi_{q_2}\) continuously to zero by
\(\Phi_{q_1}(0)=\Phi_{q_2}(0)=0\).
Fix \(t>0\), \(R>(1\vee x\vee y)\), and
\(0<\varepsilon<(1\wedge x\wedge y)\). The function \(H\) is
\(C^2\) on \([\varepsilon,R]^2\), and all state coefficients are bounded
there. Since \(0<q_i\le1\), one has \(D_i^{(q_i)}\ge0\), \(i=1,2\). By \eqref{eq:LH-D} and \eqref{mart}, one obtains 
\beqlb\label{eq:H-before-R}
\mathbb{E}\left[H(X_{t\wedge\tau_R^+\wedge\tau_\varepsilon^-},
   Y_{t\wedge\tau_R^+\wedge\tau_\varepsilon^-})\right]\ar\le\ar  
  H(x,y)
 - \mathbb{E}\left[ \int_0^{t\wedge\tau_R^+\wedge\tau_\varepsilon^-}
 \left[D_1^{(q_1)}(X_s)+\frac{a_1}{a_2}D_2^{(q_2)}(Y_s)\right] \dd s \right] \cr 
 \ar\le\ar H(x,y).
\eeqlb

\rev{By \eqref{eq:H}, one sees that}
\beqlb\label{infH}
 \inf_{\{u\vee v\ge R\}}H(u,v) \ge
 \min\left\{\frac{R^{q_1}}{q_1},\frac{a_1}{a_2}\frac{R^{q_2}}{q_2}\right\}
 \longrightarrow\infty,\qquad\rev{ R\to\infty }.
\eeqlb

On \(\{\tau_R^+\le t\wedge\tau_\varepsilon^-\}\), we have $H(X_{\tau_R^+},Y_{\tau_R^+})\ge \inf_{\{u\vee v\ge R\}}H(u,v)$.
Thus \eqref{eq:H-before-R} yields
\beqnn
\inf_{ u\vee v\ge R }H(u,v) \Pp_{x,y}(\tau_R^+\le t\wedge\tau_\varepsilon^-)
 \le H(x,y).
\eeqnn
Letting \(\varepsilon\downarrow0\), it follows that
\beqlb\label{eq:upper-exit-before-boundary}
 \Pp_{x,y}(\tau_R^+\le t,\ \tau_R^+\le \tau_0)
 \le \frac{H(x,y)}{ \inf_{ u\vee v\ge R }H(u,v)}.
\eeqlb
\vthirtyfour{Letting \(R\to\infty\) in the preceding inequality and using \eqref{infH}, we obtain} \(\Pp_{x,y}(\tau_\infty\le t,\ \tau_\infty\le\tau_0)=0\).
Since \(t\) is arbitrary, the model is source-conservative.

Letting first \(\varepsilon\downarrow0\) and then \(R\to\infty\) in the inequality \vthirtyfour{in}
\eqref{eq:H-before-R}, \vthirtyfour{source-conservativeness, Fatou's lemma, and
monotone convergence yield} \eqref{eq:concave-energy}. Moreover, let \(0\le s<t\).
On \(\{s<\tau_0\}\), similar to the discussion in \eqref{eq:H-before-R}, we have  
\beqnn
  \E\left[
 H(X_t,Y_t)
 +\int_{s\wedge\tau_0}^{t\wedge\tau_0}
 \left(D_1^{(q_1)}(X_u)+\frac{a_1}{a_2}D_2^{(q_2)}(Y_u)\right)\dd u
 \Bigm|\mathcal F_s\right] \le H(X_s,Y_s).
\eeqnn
Therefore
\(\bigl(H(X_t,Y_t)\bigr)_{t\ge0}\) is a nonnegative
supermartingale.
\end{proof}

\begin{lemma}
\label{lem:conditional-increment}
{
Assume that the model \eqref{eq:noimm-system} is source-conservative.  Fix \(M>1\) and
\(p\in(1,\alpha_1\wedge\alpha_2)\).  There exists a constant \(C_M>0\) such that,
for every almost surely finite stopping time \(T\) and every \(h\in(0,1]\),
\beqnn
 \E\left[
 \sup_{0\le u\le h}
 |X_{(T+u)\wedge\tau_M^+}-X_{T\wedge\tau_M^+}|^p
 \Bigm|\mathcal F_T\right]
 \le C_M(h^p+h^{p/2}+h).
 \eeqnn 
The same estimate holds with \(X\) replaced by \(Y\).
}
\end{lemma}

\begin{proof}
\rev{\vthirtyfour{The estimate for \(Y\) is analogous, so we give only the proof for \(X\).} For \(0\le u\le h\), we have}
{
\begin{align*}
 X_{(T+u)\wedge\tau_M^+}-X_{T\wedge\tau_M^+}
 =V(u)+M_0(u) +M_{\mathrm{1}}(u)+M_{\mathrm{2}}(u),
\end{align*}
where
\begin{align*}
 V(u):={}&\int_{T\wedge\tau_M^+\wedge\tau_0}^{(T+u)\wedge\tau_M^+\wedge\tau_0}
 \bigl(a_1X_s^{\theta_1}Y_s^{\kappa_1}-b_{10}X_s^{r_{10}}\bigr)\dd s,\\
 M_0(u):={}&\int_{T\wedge\tau_M^+\wedge\tau_0}^{(T+u)\wedge\tau_M^+\wedge\tau_0}
 \sqrt{2b_{11}X_s^{r_{11}}}\dd B_1(s),\\
 M_{\mathrm{1}}(u):={}&
 \int_{T\wedge\tau_M^+\wedge\tau_0}^{(T+u)\wedge\tau_M^+\wedge\tau_0}
 \int_0^1\int_0^{b_{12}X_{s-}^{r_{12}}}
 z\,\widetilde N_1(\dd s,\dd z,\dd v)
\end{align*}
and
\beqnn
 M_{\mathrm{2}}(u):= 
\int_{T\wedge\tau_M^+\wedge\tau_0}^{(T+u)\wedge\tau_M^+\wedge\tau_0}
\int_1^\infty\int_0^{b_{12}X_{s-}^{r_{12}}}
z\,\widetilde N_1(\dd s,\dd z,\dd v).
\eeqnn 
By H\"older's inequality and \vthirtyfour{the Burkholder--Davis--Gundy inequality, there exist constants} \(C_1,C_p>0\) such that 
\beqnn
 \E\left[\sup_{0\le u\le h}|V(u)|^p\Bigm|\mathcal F_T\right]
  \le h^{p-1}\E\left[
 \int_{T\wedge\tau_M^+\wedge\tau_0}^{(T+h)\wedge\tau_M^+\wedge\tau_0}
 \bigl|a_1X_s^{\theta_1}Y_s^{\kappa_1}-b_{10}X_s^{r_{10}}\bigr|^p\dd s
 \Bigm|\mathcal F_T\right] 
 \le C_1h^p 
\eeqnn
and
\beqnn
 \E\left[\sup_{0\le u\le h}|M_0(u)|^p\Bigm|\mathcal F_T\right]
  \le C_p\E\left[
 \left(
 \int_{T\wedge\tau_M^+\wedge\tau_0}^{(T+h)\wedge\tau_M^+\wedge\tau_0}
 2b_{11}X_s^{r_{11}}\dd s
 \right)^{p/2}\Bigm|\mathcal F_T\right] 
 \le C_1h^{p/2}.
\eeqnn
}
Moreover, \vthirtyfour{for the small-jump martingale \(M_{\mathrm{1}}(u)\), its conditional \(L^2\) estimate and Jensen's inequality give}
\beqnn
 \E\left[\sup_{0\le u\le h}|M_{\mathrm{1}}(u)|^p\mid\mathcal F_T\right] \ar\le\ar 
 \left(\E\left[\sup_{0\le u\le h}|M_{\mathrm{1}}(u)|^2\mid\mathcal F_T\right]\right)^{p/2}\cr 
 \ar\le\ar C_M\left(h\int_0^1z^2\mu_1(\dd z)\right)^{p/2}\le C_2h^{p/2}
\eeqnn
for some constants $C_M, C_2 > 0$.
For the large-jump martingale $M_{\mathrm{2}}(u)$, we have $\int_1^\infty z^p \mu_1(\dd z) < \infty$  since $p \in (1, \alpha_1\wedge\alpha_2)$. By \vthirtyfour{the Burkholder--Davis--Gundy
inequality} and $ (\sum_j a_j)^{p/2}
\le \sum_j a_j^{p/2}$ for $a_j \ge 0$, one obtains
\beqnn
 \E\left[
	\sup_{0\le u\le h}|M_{\mathrm{2}}(u)|^p
	\Bigm|\mathcal F_T\right]
	\ar\le\ar  C_3\E\left[
	\left(
	\int_{T\wedge\tau_M^+\wedge\tau_0}^{(T+h)\wedge\tau_M^+\wedge\tau_0}
	\int_1^\infty\int_0^{b_{12}X_{s-}^{r_{12}}}
	z^2N_1(\dd s,\dd z,\dd v)
	\right)^{p/2}
	\Bigm|\mathcal F_T\right]\cr
	\ar\le\ar C_3\E\left[
	\int_{T\wedge\tau_M^+\wedge\tau_0}^{(T+h)\wedge\tau_M^+\wedge\tau_0}
	\int_1^\infty\int_0^{b_{12}X_{s-}^{r_{12}}}
	z^pN_1(\dd s,\dd z,\dd v)
	\Bigm|\mathcal F_T\right]\cr
	\ar\le\ar C_3b_{12}M^{r_{12}}h
	\int_1^\infty z^p\mu_1(\dd z)
\eeqnn
for some constant $C_3 > 0$. The result follows. 
\end{proof}

\begin{lemma} 
\label{lem:selfreg-vanishing}
{
Assume that $0 < q_1, q_2 \le 1$. If \(b_{10},b_{20}>0\), then $(X_t,Y_t)\rightarrow(0,0)$ on $\{\tau_0=\infty\}.$
}
\end{lemma}

\begin{proof}
By Lemma~\ref{lem:concave-Lyapunov},
\(\bigl(H(X_t,Y_t)\bigr)_{t\ge0}\) is a nonnegative supermartingale.
The supermartingale convergence theorem therefore shows that
\rev{\(H(X_t,Y_t)\) converges almost surely to a finite limit as
\(t\to\infty\); see, e.g.,
\vthirtyfour{Revuz and Yor}~\cite[Chapter~II, Corollary~2.11]{RevuzYor1999}.}
 Since $t \mapsto H(X_t,Y_t)$ has c\`adl\`ag paths and every c\`adl\`ag function is
bounded on compact time intervals, it follows that
$\sup_{t\ge0}H(X_t,Y_t)<\infty$ almost surely. \vthirtyfour{Moreover, since} $X_t^{q_1}\le q_1H(X_t,Y_t)$ and $Y_t^{q_2}\le\frac{a_2 q_2}{a_1}H(X_t,Y_t)$, \vthirtyfour{it follows that}
\begin{equation}
 \sup_{t\ge0}(X_t\vee Y_t)<\infty
 \qquad\rev{\text{almost surely}.}
 \label{eq:path-bounded}
\end{equation}
Furthermore, \vthirtyfour{letting \(t\to\infty\) in \eqref{eq:concave-energy} and using \(D_i^{(q_i)}(z)\ge b_{i0}z^{q_i+r_{i0}-1}\) together with monotone convergence, we obtain}
\begin{equation}
	\int_0^{\tau_0}
	\left(X_s^{q_1+r_{10}-1}+Y_s^{q_2+r_{20}-1}\right)\dd s<\infty
	\qquad\text{almost surely}.
	\label{eq:selfreg-occupation}
\end{equation}
Fix \(M>1\), \(0<\delta<M/2\), and
\(p\in(1,\alpha_1\wedge\alpha_2)\).
 Let $T$ be any almost surely finite stopping time. On $\{X_T\ge2\delta,\ T\le \tau_M^+\}$, we have $X_{T\wedge\tau_{M}^+} = X_T \ge 2\delta$. If $\inf_{0 \le u \le h}X_{(T+u)\wedge \tau_M^+} < \delta$, then there exists a $u \in [0, h]$ such that $X_{(T+u)\wedge \tau_M^+} < \delta$. It follows that $|X_{(T+u)\wedge\tau_M^+} - X_{T\wedge\tau_M^+}| = X_T - X_{(T+u)\wedge\tau_M^+} > \delta.$
 Hence,
 \beqnn
\ar\ar \left\{\inf_{0 \le u \le h}X_{(T+u)\wedge\tau_M^+} < \delta, T \le \tau_M^+, X_T \ge 2\delta\right\}\cr 
 \ar\ar\qquad \subseteq \left\{\sup_{0 \le u \le h}|X_{(T+u)\wedge\tau_M^+} - X_{T\wedge\tau_M^+}| \ge \delta, T \le \tau_M^+, X_T \ge 2\delta\right\}.
 \eeqnn 
Notice that $\{T \le \tau_M^+, X_T \ge 2\delta\} \in \mathcal{F}_T$. Then, by Markov's inequality and \vthirtyfour{choosing \(h>0\) sufficiently small} in Lemma \ref{lem:conditional-increment}, one can obtain that
\beqnn
\ar\ar \Pp\left(
\inf_{0\le u\le h}X_{(T+u)\wedge\tau_M^+}<\delta, T\le \tau_M^+,\ X_T\ge2\delta
\Bigm|\mathcal F_T\right)\cr
\ar\ar\qquad\le  \Pp\left( 
\sup_{0\le u\le h}|X_{(T+u)\wedge\tau_M^+}-X_{T\wedge\tau_M^+}|\ge\delta, T\le \tau_M^+,\ X_T\ge2\delta
\Bigm|\mathcal F_T\right)\cr
\ar\ar\qquad= 1_{\{ T<\tau_M^+,\ X_T\ge2\delta\}}\Pp\left( 
\sup_{0\le u\le h}|X_{(T+u)\wedge\tau_M^+}-X_{T\wedge\tau_M^+}|\ge\delta\Bigm|\mathcal F_T\right)\cr
 \ar\ar\qquad\le 1_{\{ T<\tau_M^+,\ X_T\ge2\delta\}}\delta^{-p}\E\left[
\sup_{0\le u\le h}|X_{(T+u)\wedge\tau_M^+}-X_T|^p
\Bigm|\mathcal F_T\right]\le\frac12 1_{\{ T<\tau_M^+,\ X_T\ge2\delta\}}.
\eeqnn 
Equivalently,
\beqlb\label{eq:stay-positive}
 \Pp\left(
 \inf_{0\le u\le h}X_{(T+u)\wedge\tau_M^+}\ge\delta
 \Bigm|\mathcal F_T\right)\ge\frac12 \qquad \text{on}\ \{T \le \tau_M^+, X_T \ge 2\delta\}.
\eeqlb
Suppose \(\limsup_{t\to\infty}X_t>2\delta\). 
 Define disjointly spaced stopping times
\[
 T_1:=\inf\{t\ge0:X_t\ge2\delta\},
 \qquad
 T_{k+1}:=\inf\{t\ge T_k+h:X_t\ge2\delta\}.
\]
Then every \(T_k\) is finite on $\{\tau_0=\infty,\ \tau_M^+=\infty,\
\limsup_{t\to\infty}X_t>2\delta\}$.  Let
\beqnn
E_k:=\left\{T_k<\infty,\
 \inf_{0\le u\le h}X_{(T_k+u)\wedge\tau_M^+}\ge\delta\right\}.
\eeqnn
Set \(\mathcal G_0:=\mathcal F_{T_1}\) and
\(\mathcal G_k:=\mathcal F_{T_{k+1}}\), \(k\ge1\).  Since
\(T_{k+1}\ge T_k+h\), one has
\(E_k\in\mathcal F_{T_k+h}\subseteq\mathcal G_k\).  Thus \((E_k)\) is adapted to the
discrete filtration \((\mathcal G_k)\).  On
\(\{\tau_0=\infty,\ \tau_M^+=\infty,\
\limsup_{t\to\infty}X_t>2\delta\}\), one has
\(T_k<\infty\), \(T_k<\tau_M^+\), and \(X_{T_k}\ge2\delta\) for every
\(k\).  Thus, by \eqref{eq:stay-positive},
\[
 \sum_{k=1}^\infty
 \Pp(E_k\mid\mathcal G_{k-1})
 =\sum_{k=1}^\infty\Pp(E_k\mid\mathcal F_{T_k})=\infty
\]
on
\(\{\tau_0=\infty,\ \tau_M^+=\infty\}\).
The conditional Borel--Cantelli lemma therefore implies that infinitely many
\(E_k\) occur almost surely on this event.  Since the intervals
\([T_k,T_k+h]\) have disjoint interiors, on \(\{\tau_0=\infty,\ \tau_M^+=\infty\}\),
\beqnn
 \int_0^\infty X_s^{q_1+r_{10}-1}\dd s
 \ge\sum_{k:E_k\text{ occurs}}
 \int_{T_k}^{T_k+h}X_s^{q_1+r_{10}-1}\dd s\ge\sum_{k:E_k\text{ occurs}}
 h\min_{\delta\le u\le M}u^{q_1+r_{10}-1}=\infty,
\eeqnn
contradicting \eqref{eq:selfreg-occupation}.  Hence
\[
 \limsup_{t\to\infty}X_t\le2\delta
 \quad\text{on }\{\tau_0=\infty,\ \tau_M^+=\infty\}.
\]
\vthirtyfour{Letting \(\delta\downarrow0\), we obtain}
\(X_t\to0\) on \(\{\tau_0=\infty,\ \tau_M^+=\infty\}\). Repeating the preceding argument with \(Y\) in place of \(X\), using the
\(Y\)-version of \eqref{eq:stay-positive} and the finiteness of
\(\int_0^{\tau_0}Y_s^{q_2+r_{20}-1}\dd s\) from
\eqref{eq:selfreg-occupation}, yields \(Y_t\to0\) on
\(\{\tau_0=\infty,\ \tau_M^+=\infty\}\).  The result follows \vthirtyfour{by letting \(M\to\infty\) and using \eqref{eq:path-bounded}}.
\end{proof}

\begin{proof}[Proof of Theorem~\ref{thm:concave-cancellation}]
{
Lemma~\ref{lem:concave-Lyapunov} proves source-conservativeness,
\eqref{eq:concave-energy}, and the supermartingale assertion.  Assume
additionally that \(b_{10},b_{20}>0\).  On \(\{\tau_0<\infty\}\), source
stopping gives
\((X_t,Y_t)=(X_{\tau_0},Y_{\tau_0})\) for \(t\ge\tau_0\), and at least
one coordinate of this limit is zero.  On \(\{\tau_0=\infty\}\), both
coordinates converge to zero by Lemma~\ref{lem:selfreg-vanishing}. This proves all assertions.
}
\end{proof}

\subsection[Proof of Theorem~\ref{thm:dissipativity}]{\vnew{Proof of Theorem~\ref{thm:dissipativity}}}

\begin{proof}[Proof of Theorem~\ref{thm:dissipativity}]
{
For \(i=1,2\), \eqref{eq:dissipative-b} gives
\(q_i+r_{i0}-1>0\).  Recall \(D_i^{(q_i)}\) is given by \eqref{eq:Dip}.
By Lemma \ref{lem:stable-power}, 
{
\[
 -D_i^{(q_i)}(z)
 =-b_{i0}z^{q_i+r_{i0}-1}
 +(q_i-1)b_{i1}z^{q_i+r_{i1}-2}
 +b_{i2}j_{\alpha_i}(q_i)z^{q_i+r_{i2}-\alpha_i}.
\]}
If \(0<q_i\le1\), then
\(j_{\alpha_i}(q_i)\le0\) by \eqref{eq:j-alpha}, and hence $D_i^{(q_i)}(z)\ge b_{i0}z^{q_i+r_{i0}-1}.$
If \(1<q_i<\alpha_i\), conditions
\eqref{eq:dissipative-b}--\eqref{eq:dissipative-c} show that each active
correction exponent $q_i + r_{i1} - 2$ or $q_i + r_{i2} - \alpha_i$ is nonnegative and strictly smaller than
\(q_i+r_{i0}-1\).  Then there exists
a constant \(C_i>0\) such that
\[
 D_i^{(q_i)}(z)
 \ge \frac{b_{i0}}2z^{q_i+r_{i0}-1}-C_i,
 \qquad z > 0.
\]
\vthirtyfour{Equation~\eqref{eq:LH-D} gives}
\beqlb\label{eq0821b}
 \cL H(x,y)
 =-D_1^{(q_1)}(x)-\frac{a_1}{a_2}D_2^{(q_2)}(y)
 \le C_1+\frac{a_1}{a_2}C_2
 -\frac{b_{10}}2x^{q_1+r_{10}-1}
 -\frac{a_1b_{20}}{2a_2}y^{q_2+r_{20}-1}.
\eeqlb
 This proves
the first assertion with $C= C_1+\frac{a_1}{a_2}C_2$ and $c = \frac{b_{10}}2 \wedge \frac{a_1b_{20}}{2a_2}$.

Fix \(R>1\vee x\vee y\) and
\(0<\varepsilon<1\wedge x\wedge y\). \vthirtyfour{Applying \eqref{mart} to \(H\) and using \eqref{eq0821b}, we obtain}
\beqlb\label{eq:dissipative-localized}
  \E H(X_{t\wedge\tau_R^+\wedge\tau_\varepsilon^-},
       Y_{t\wedge\tau_R^+\wedge\tau_\varepsilon^-}) +c\E\int_0^{t\wedge\tau_R^+\wedge\tau_\varepsilon^-}
 \left(X_s^{q_1+r_{10}-1}+Y_s^{q_2+r_{20}-1}\right)\dd s
 \le H(x,y)+Ct. 
\eeqlb
Similar to \eqref{eq:upper-exit-before-boundary}, one sees that
\beqnn
\Pp_{x,y}(\tau_R^+\le t,\ \tau_R^+\le \tau_0)
\le \frac{H(x,y)+Ct}{\inf_{u\vee v \ge R}H(u, v)}.
\eeqnn 
 \vthirtyfour{Letting \(R\to\infty\) and using \eqref{infH}, we obtain}
\[
 \Pp_{x,y}(\tau_\infty\le t,\ \tau_\infty\le\tau_0)=0.
\]
Since \(t\) is arbitrary, the model \eqref{eq:noimm-system} is source-conservative.

Finally, taking \(t=T\), \(\varepsilon\downarrow0\) and
\(R\to\infty\) in \eqref{eq:dissipative-localized}, and applying Fatou's lemma, we have 
\begin{align*}
 &c\E\int_0^{T\wedge\tau_0}
 \left(X_s^{q_1+r_{10}-1}+Y_s^{q_2+r_{20}-1}\right)\dd s
 \le H(x,y)+CT.
\end{align*}
Then the last assertion is obtained by dividing by $cT$.
}
\end{proof}

\subsection[Proofs of Theorem~\ref{thm:boundary-clock} and \vthirtyfour{Corollary~\ref{cor:lower-exposure}}]{\vnew{Proofs of Theorem~\ref{thm:boundary-clock} and \vthirtyfour{Corollary~\ref{cor:lower-exposure}}}}
\label{sec:boundary-extinction}

\begin{proof}[Proof of Theorem~\ref{thm:boundary-clock}]
For \(\varepsilon>0\), set $
 \tau_\varepsilon^{X,-}:=\inf\{t\ge0:X_t\le\varepsilon\}$ and $
 \tau_\varepsilon^{Y,-}:=\inf\{t\ge0:Y_t\le\varepsilon\}.$
Since all jumps are nonnegative, every downward crossing is continuous.  Thus, on the event that the relevant coordinate reaches zero before source stopping by the other coordinate,
\(
 \tau_\varepsilon^{X,-}\uparrow\tauX
\) or
\(
 \tau_\varepsilon^{Y,-}\uparrow\tauY
\), respectively.

We first prove parts \textup{(i)} and \textup{(ii)}.
Fix \(q>0\).  For \(f(x,y)=x^{-q}\), we have $
 f_x'(x, y)=-qx^{-q-1}$ and $
 f_{xx}''(x, y)=q(q+1)x^{-q-2}.$
After the substitution \(z=xw\) in the jump term,
\[
 \int_0^\infty\bigl[(x+z)^{-q}-x^{-q}+qx^{-q-1}z\bigr]\mu_1(\dd z)
 =x^{-q-\alpha_1}K_{\alpha_1,q},
\]
where
{
\[
 K_{\alpha,q}:=\int_0^\infty\bigl[(1+v)^{-q}-1+qv\bigr]\mu_\alpha(\dd v) = -qj_\alpha(-q)
\]}
by Lemma \ref{lem:stable-power}.
Therefore, by \eqref{eq:noimm-generator},
{
\begin{align*}
 \frac{\cL f(x,y)}{f(x,y)}
 ={}&-qa_1x^{\theta_1-1}y^{\kappa_1}
 +qb_{10}x^{r_{10}-1}
 +q(q+1)b_{11}x^{r_{11}-2}
 +b_{12}K_{\alpha_1,q}x^{r_{12}-\alpha_1}.
\end{align*}}
If \(r_1\ge0\), then for each $N>1$, there exists a constant $C_N > 0$ such that 
{
\[
 \cL f(x,y)\le C_Nf(x,y), \qquad 0 < x, y \le N.
\]}

For \(N>1\vee x\vee y\) and \(0<\varepsilon<x\),
\vthirtyfour{\eqref{mart} and Gronwall's inequality give}
\[
 \E_{x,y}f(X_{t\wedge\tau_\varepsilon^{X,-}\wedge\tau_N^+\wedge\tau_0},
 Y_{t\wedge\tau_\varepsilon^{X,-}\wedge\tau_N^+\wedge\tau_0})
 \le x^{-q}\e^{C_Nt}.
\]
On \(\{\tau_\varepsilon^{X,-}\le t,\ \tau_\varepsilon^{X,-}<\tau_N^+\wedge\tau_0\}\), by the continuity of the downward crossing, we have \(X_{\tau_\varepsilon^{X,-}}=\varepsilon\).  Thus
\beqnn
 \Pp_{x,y}\bigl(\tau_\varepsilon^{X,-}\le t,
 \ \tau_\varepsilon^{X,-}<\tau_N^+\wedge\tau_0\bigr)
 \le\left(\frac\varepsilon x\right)^q\e^{C_Nt}.
\eeqnn
Letting
\(\varepsilon\downarrow0\), and then \(N\to\infty\), we get $\mathbb{P}(\tau_0^X\le t, \tau_0^X \le \tau_\infty\wedge\tau_0) = 0$, which gives
\beqnn
\mathbb{P}(\tau_0^X < \infty, \tau_0^X\le \tau_\infty\wedge\tau_0) = 0.
\eeqnn 
If \(\tauX<\infty\), then \(\tauX=\tau_0\). 
Part \textup{(i)} follows \vthirtyfour{from the source-conservativeness of the model} \eqref{eq:noimm-system}.

For part \textup{(ii)}, take \(g(x,y)=y^{-q}\).  Similarly,
{
\begin{align*}
 \frac{\cL g(x,y)}{g(x,y)}
 ={}&qa_2 x^{\kappa_2}y^{\theta_2-1}
 +qb_{20}y^{r_{20}-1}
 +q(q+1)b_{21}y^{r_{21}-2}
 +b_{22}K_{\alpha_2,q}y^{r_{22}-\alpha_2}.
\end{align*}}
If \(\theta_2\ge1\), \(r_2\ge0\), and \(0<x,y\le N\), the right-hand side is bounded above by a deterministic constant \(\widetilde C_N\).  Repeating the preceding stopped argument with \(\tau_\varepsilon^{Y,-}\) yields
\[
 \Pp_{x,y}\bigl(\tau_\varepsilon^{Y,-}\le t,
 \ \tau_\varepsilon^{Y,-}<\tau_N^+\wedge\tau_0\bigr)
 \le\left(\frac\varepsilon y\right)^q\e^{\widetilde C_Nt}.
\]
The same limits prove part \textup{(ii)}.

For part \textup{(iii)}, let $h(x, y)=y^{1-\theta_2}$. For $y>0$, we have $ h_y'(x,y) =(1-\theta_2)y^{-\theta_2},$
   and $h_{yy}''(x, y)=-\theta_2(1-\theta_2)y^{-1-\theta_2}\le0,$
which implies
\beqnn 
 (y+z)^{1-\theta_2}-y^{1-\theta_2}
 -(1-\theta_2)y^{-\theta_2}z\le0.
\eeqnn 
Then by \eqref{eq:noimm-generator},
\beqlb\label{eq:power-clock-generator}
 \cL h(x, y)
  \le -(1 - \theta_2)a_2 x^{\kappa_2}.
\eeqlb
For $N>1\vee x\vee y$ and $0<\varepsilon<y$, \vthirtyfour{\eqref{mart} and \eqref{eq:power-clock-generator} give}
\[
 \E Y_{t\wedge\tau_N^+\wedge\tau_\varepsilon^{Y,-}\wedge\tau_0}^{1-\theta_2}
 +(1-\theta_2)a_2\E
 \int_0^{t\wedge\tau_N^+\wedge\tau_\varepsilon^{Y,-}\wedge\tau_0}
 X_s^{\kappa_2}\dd s
 \le y^{1-\theta_2}.
\]
\vthirtyfour{Letting \(N\to\infty\) and then \(\varepsilon\downarrow0\), Fatou's lemma and source-conservativeness yield}
\begin{equation}
 \E Y_{t\wedge\tau_0}^{1-\theta_2}
 +(1-\theta_2)a_2\E\int_0^{t\wedge\tau_0}X_s^{\kappa_2}\dd s
 \le y^{1-\theta_2}.
 \label{eq:finite-basic-clock}
\end{equation}
Since the source-stopped process is frozen at \(\tau_0\), and \(X\) is
zero if it reaches the boundary first, then $\int_0^{\tau_0}X_s^{\kappa_2}\dd s
 =\int_0^{\tauY}X_s^{\kappa_2}\dd s.$ \eqref{eq:basic-clock} follows by letting $t\to\infty$.

On \(\{\tauY=\infty\}\), the random variable \(\int_0^{\tauY}X_s^{\kappa_2}\dd s\) is finite almost surely by \eqref{eq:basic-clock}; taking the contrapositive, it proves \eqref{eq:exposure-implies-extinction}. { Finally, on the event $\{\tauY>T,\
\int_0^T X_s^{\kappa_2}\dd s\ge L \},$
either $\tau_0>T$, or $X$ reaches zero \vthirtyfour{by time} \(T\). In the latter case, source stopping gives $X_s=0$ for $s\ge\tau_0$. Consequently, on $\{\tauY>T,\
\int_0^T X_s^{\kappa_2}\dd s\ge L \},$
\[
\int_0^{T\wedge\tau_0}X_s^{\kappa_2}\dd s
=\int_0^T X_s^{\kappa_2}\dd s\ge L.
\]
Hence, by Markov's inequality and \eqref{eq:finite-basic-clock},
\beqnn
	\Pp_{x,y}\left(
	\tauY>T,\
	\int_0^T X_s^{\kappa_2}\dd s\ge L\right) \le
	\frac1L\E_{x,y}
	\int_0^{T\wedge\tau_0}X_s^{\kappa_2}\dd s
	\le\frac{y^{1-\theta_2}}
	{a_2(1-\theta_2)L}.
\eeqnn
The result follows.}
\end{proof}

\begin{proof}[Proof of \vthirtyfour{Corollary~\ref{cor:lower-exposure}}]
	Observe that
	\[
	\int_0^\infty X_s^{\kappa_2}\dd s
	\ge c_0^{\kappa_2}\int_{T_0}^\infty(1+s)^{-\beta\kappa_2}\dd s=\infty
	\]
	when \(\beta\kappa_2\le1\).  The inclusion \eqref{eq:exposure-implies-extinction} completes the proof.
\end{proof}

\subsection[Proof of Theorem~\ref{thm:kappa-le-r}]{\vnew{Proof of Theorem~\ref{thm:kappa-le-r}}}

\begin{proof}[Proof of Theorem~\ref{thm:kappa-le-r}]
 We couple an auxiliary process
$\overline{Y}$ to the same noises as $Y$,
but replace $X$ in its drift by  
$\Xlow$ given by \eqref{eq:lower-envelope}. Define $\tau_0^{\overline{Y}} := \inf\{t\ge0: \overline{Y}_t =0\}.$ \vthirtyfour{Lemma~\ref{lem:lower-comparison} gives} $\Xlow_t \le X_t$ almost surely. \vthirtyfour{Under the assumptions of this theorem, \cite[Theorem~1.7 and Example~1.12]{RenXiongYangZhou2022} yields}
$\Pp_{x,y}(\tau_0^{\overline{Y}} <\infty)=1.$
By the comparison theorem
\cite[Proposition~3.6]{RenXiongYangZhou2022}, we have almost surely $Y_t\le \overline{Y}_t$ until either process reaches zero. Therefore $\tauY\le\tau_0^{\overline{Y}}<\infty$ almost surely. The result follows.  
\end{proof}

\subsection[Proof of Theorem~\ref{thm:FK-rate}]{\vnew{Proof of Theorem~\ref{thm:FK-rate}}}

\begin{proof}[Proof of Theorem~\ref{thm:FK-rate}]
{ We first prove part \textup{(i)}.
Fix \(R>1\vee x\vee y\), \(\varepsilon\in(0,y)\), and let
\beqnn
\sigma_{R,\varepsilon} := \tau_\varepsilon^{Y,-}\wedge\tau_0
 \wedge\inf\{t\ge0:X_t\vee Y_t\vee G_{2,t}\vee G_{2,t}^{-1}\ge R\}.
\eeqnn 
\vthirtyfour{Here} $G_{2,t}$ satisfies \eqref{eq:G2}. By \vthirtyfour{It\^o's formula}, 
\beqlb\label{eq0820} 
 G_{2,t}^{\theta_2-1}Y_t^{1-\theta_2}
 =y^{1-\theta_2}-(1-\theta_2)a_2
 \int_0^tG_{2,u}^{\theta_2-1}X_u^{\kappa_2}\dd u, \quad t \le \sigma_{R,\varepsilon}.
\eeqlb
Let \(R\to\infty\) and \(\varepsilon\downarrow0\).  Source-conservativeness and the c\`adl\`ag paths give \eqref{eq:FK-identity} for every $t<\tau_0$.

The right-hand side of \eqref{eq0820} is continuous and nonincreasing,
and it is strictly positive for every \(t<\tau_0\).  If
\(\tauY<\infty\), then \(\tauY=\tau_0\).  Moreover, by
Lemma \ref{lem:geometric-factor},
\(G_{2,t}\to G_{2,\tauY-}\in(0,\infty)\) as \(t\uparrow\tauY\).
Letting \(t\uparrow\tauY\) in \eqref{eq0820} gives
\[
\int_0^{\tauY}
G_{2,u}^{\theta_2-1}X_u^{\kappa_2}\dd u
=
\frac{y^{1-\theta_2}}{(1-\theta_2)a_2}.
\]
Next, suppose that \(\tauY=\infty\) and \(\tauX<\infty\).  Then
\(\tau_0=\tauX\) and \(Y_{\tau_0}>0\).  \vthirtyfour{Letting
\(t\uparrow\tau_0\) in \eqref{eq0820} yields}
\beqlb\label{eq0824}
\int_0^{\tau_0}
G_{2,u}^{\theta_2-1}X_u^{\kappa_2}\dd u
=
\frac{
	y^{1-\theta_2}
	-G_{2,\tau_0-}^{\theta_2-1}Y_{\tau_0}^{1-\theta_2}}
{(1-\theta_2)a_2}
<
\frac{y^{1-\theta_2}}{(1-\theta_2)a_2}.
\eeqlb
 If
\(\tau_0=\infty\), by \eqref{eq:FK-identity},  replacing $\tau_0$ by $t$ in \eqref{eq0824}, it holds for all $t \ge 0$.  Those cases prove \eqref{eq:FK-hitting}.

On \(\{\tauY=\infty\}\), \eqref{eq0824} holds when
\(\tau_0<\infty\); when \(\tau_0=\infty\), \vthirtyfour{monotone convergence and
\eqref{eq:FK-identity} give}
\[
\int_0^\infty
G_{2,u}^{\theta_2-1}X_u^{\kappa_2}\dd u
\le
\frac{y^{1-\theta_2}}{(1-\theta_2)a_2}.
\]
This proves \eqref{eq:FK-survival-budget}.}

For part \textup{(ii)}, we assume \(\{\tauY=\infty\}\).
Let \(\Xlow\) be the solution of \eqref{eq:lower-envelope}. \vthirtyfour{Proposition~\ref{prop:lower-log-rate} and
Lemma~\ref{lem:lower-comparison} show that} \(X_t\ge\Xlow_t>0\) for all $t$.  Thus
\(\tauX=\infty\) and \(\tau_0=\infty\) on the event under consideration.  
Therefore, for every \(t>0\),
\begin{equation}
 \int_0^tG_{2,u}^{\theta_2-1}X_u^{\kappa_2}\dd u
 \ge\int_0^tG_{2,u}^{\theta_2-1}\Xlow_u^{\kappa_2}\dd u.
 \label{eq:weighted-lower-clock}
\end{equation}
Combining Lemma \ref{lem:geometric-factor} with
Proposition~\ref{prop:lower-log-rate}, we obtain
\beqlb\label{eq0824a} 
 \lim_{t\to\infty}\frac1t
 \log\bigl(G_{2,t}^{\theta_2-1}\Xlow_t^{\kappa_2}\bigr)
 =(1-\theta_2)b_2-\kappa_2b_1\1_{\{r_1=0\}}.
\eeqlb
If the right-hand side of \eqref{eq0824a} is positive,
the integrand in
\eqref{eq:weighted-lower-clock} is eventually bounded below by
\(\e^{\delta t}\) for some \(\delta>0\).  The weighted integral therefore
diverges, contradicting \eqref{eq:FK-survival-budget}.  This proves part
\textup{(ii)(a)}.

For part \textup{(ii)(b)}, by the assumptions and
Proposition~\ref{prop:relative-jump}, one sees that
\(\Xlow_t=xG_{1,t}\), with \(G_1\) independent of \(G_2\).  Set 
\[
 \xi_t:=(1-\theta_2)\log G_{2,t}-\kappa_2\log G_{1,t}.
\]
Then \(\xi := (\xi_t)_{t \ge 0}\) is an integrable L\'evy process. By
Lemma \ref{lem:geometric-factor}, we have
\(\E\xi_1=\kappa_2b_1-(1-\theta_2)b_2=0\) and $\lim_{t \rightarrow \infty}\frac{1}{t}\xi_t = 0$ almost surely. \vthirtyfour{The criterion of
Bertoin and Yor}~\cite[Theorem~1]{BertoinYor2005} therefore gives
\[
 \int_0^\infty
 G_{2,t}^{\theta_2-1}\Xlow_t^{\kappa_2}\dd t
 =x^{\kappa_2}\int_0^\infty\e^{-\xi_t}\dd t
 =\infty
 \qquad\text{almost surely}.
\]
Together with \eqref{eq:weighted-lower-clock}, this contradicts
\eqref{eq:FK-survival-budget} on \(\{\tauY=\infty\}\).  Therefore
\(\Pp_{x,y}(\tauY<\infty)=1\), proving part \textup{(ii)(b)}.

Finally, if \(r_1>0\), the right-hand side of
\eqref{eq0824a} equals
\((1-\theta_2)b_2>0\).  The argument used for part~\textup{(ii)(a)}
therefore proves part~\textup{(ii)(c)}. 
\end{proof}

\subsection[Proof of Theorem~\ref{thm:critical-transition}]{\vnew{Proof of Theorem~\ref{thm:critical-transition}}}

\begin{proof}[Proof of Theorem~\ref{thm:critical-transition}] 
	{
(i). We first consider $t < \tau_0\wedge\tau_\infty$.   The localized calculation used in the
proof of Theorem~\ref{thm:FK-rate}\textup{(i)} gives
\eqref{eq:FK-identity} there.  In particular, it is easy to see that \(Y_t\le yG_{2,t}\).
By Proposition~\ref{prop:relative-jump}, the first equation of \eqref{eq:noimm-system} has the
 following realization: 
\[
\dd X_t=X_{t-}\left[(a_1Y_t^{\kappa_1}-b_{10})\dd t
+\sqrt{2b_{11}}\dd B_1(t)
+\int_0^\infty v\,\widetilde M_1(\dd t,\dd v)\right].
\]
\vthirtyfour{Then It\^o's formula and \eqref{eq:G2} give}
	\beqnn 
		\frac{X_t}{G_{1,t}} 
		\ar =\ar x+
		\int_0^t\frac{1}{G_{1,s-}}\,\dd X_s
		-\int_0^t\frac{X_{s-}}{G_{1,s-}^2}\,\dd G_{1,s} +\int_0^t\frac{X_{s-}}{G_{1,s-}^3}\,
		\dd\langle G_1^c\rangle_s\cr
		\ar\ar -\int_0^t\frac{1}{G_{1,s-}^2}\,
		\dd\langle X^c,G_1^c\rangle_s  +\sum_{0<s\le t}\Biggl[
		\frac{X_s}{G_{1,s}}
		-\frac{X_{s-}}{G_{1,s-}}
		-\frac{\Delta X_s}{G_{1,s-}}
		+\frac{X_{s-}\Delta G_{1,s}}{G_{1,s-}^2}
		\Biggr]\\
		\ar=\ar x + 
		a_1\int_0^t
		Y_s^{\kappa_1}\frac{X_s}{G_{1,s}}\,\dd s.
	\eeqnn 
{
	Consequently, \eqref{eq:critical-X-factor} follows for $t < \tau_0\wedge\tau_\infty$. Together with \(Y_t\le yG_{2,t}\), this yields
		\begin{equation}
			xG_{1,t}\le X_t
			\le xG_{1,t}\exp\left(
			a_1y^{\kappa_1}\int_0^tG_{2,u}^{\kappa_1}\dd u
			\right)
			\label{eq:critical-bounds}
		\end{equation}
	holds for $t<\tau_0\wedge\tau_\infty.$
	For every constant \(T> 0\), we have $
	0<\inf_{0\le t\le T}G_{i,t}
	\le\sup_{0\le t\le T}G_{i,t}<\infty$ for $i = 1, 2$. Then $\sup_{0\le t<T\wedge\tau_0\wedge\tau_\infty}
	(X_t\vee Y_t)<\infty$ almost surely. Recall that \(T\) is arbitrary.
	By the definition
	\(\tau_\infty=\lim_{R\uparrow\infty}\tau_R^+\), we have
	\[
	\Pp_{x,y}(\tau_\infty<\infty,\ \tau_\infty\le\tau_0)=0,
	\]
	so the model is source-conservative. Hence $\tau_0 < \tau_\infty$ if $\tau_0<\infty$. Letting \(t\uparrow\tau_0\) in the lower bound
	of \eqref{eq:critical-bounds}, we obtain $X_{\tau_0}\ge xG_{1,\tau_0-}>0.$
	Thus \(Y\), rather than \(X\), reaches the first boundary. Therefore
	\(\tau_0=\tauY\) almost surely, and the factorization above proves
	\eqref{eq:critical-X-factor} and {eq:critical-bounds} for every \(t<\tauY\). Finally, \eqref{eq:FK-identity}, \eqref{eq:FK-hitting} and
	\eqref{eq:FK-survival-budget} follow from Theorem~\ref{thm:FK-rate}\textup{(i)}. This proves part~\textup{(i)}.
	}

\textup{(ii)}. For \(t<\tauY\), by the lower bound in
\eqref{eq:critical-bounds}, we have
\begin{equation}
 \int_0^tG_{2,u}^{\theta_2-1}X_u^{\kappa_2}\dd u
 \ge x^{\kappa_2}\int_0^t
 G_{2,u}^{\theta_2-1}G_{1,u}^{\kappa_2}\dd u.
 \label{eq:critical-clock-lower}
\end{equation}
Set $ \xi_t:=\kappa_2\log G_{1,t}-(1-\theta_2)\log G_{2,t}.$
By Proposition~\ref{prop:relative-jump}
and \Cref{lem:geometric-factor}, one sees that  \(\xi\) is a L\'evy process and
\[
 \E|\xi_1|<\infty,\qquad
 \E\xi_1=(1-\theta_2)b_2-\kappa_2b_1\ge0.
\]
{Consequently, by  \cite[Theorem~36.5]{Sato1999}, $\lim_{t\to\infty} \xi_t/t =\E\xi_1\ge0$ almost surely.
	If \(\xi\equiv0\), then $\int_0^\infty G_{2,u}^{\theta_2-1}G_{1,u}^{\kappa_2}\dd u=\int_0^\infty\e^{\xi_u}\dd u =\int_0^\infty1\dd u=\infty.$ Suppose now that \(\xi\) is nontrivial. By \cite[Theorem~1]{BertoinYor2005}, it yields}
\[
 \int_0^\infty G_{2,u}^{\theta_2-1}G_{1,u}^{\kappa_2}\dd u
 =\int_0^\infty\e^{\xi_u}\dd u=\infty
\quad\text{almost surely}.
\]
On \(\{\tauY=\infty\}\), this contradicts
\eqref{eq:FK-survival-budget}.  This proves part~\textup{(ii)}.}

  \textup{(iii)}. Recall that $b_1, b_2 > 0$. By Lemma \ref{lem:geometric-factor}, for \(i=1,2\), almost surely, we have
\begin{equation}
	\lim_{t\to\infty}\frac1t\log G_{i,t}
	=-b_i \quad \text{and} \quad
	\int_0^\infty G_{i,t}^q\dd t<\infty\quad\vthirtyfour{\text{for any }q>0}.
	\label{eq:Gi-rates}
\end{equation}
Similarly, we define $\eta_t:=\kappa_2\log G_{1,t}
-(1-\theta_2)\log G_{2,t}$, which is a L\'{e}vy process. By \eqref{eq:Gi-rates}, it follows that $\lim_{t \rightarrow \infty}\eta_t/t =
-\kappa_2b_1+(1-\theta_2)b_2
=-\delta<0$ almost surely.
 It follows from \cite[Theorem~1]{BertoinYor2005} that
\beqlb\label{eq0821}
 \int_0^\infty G_{2,u}^{\theta_2-1}G_{1,u}^{\kappa_2}\dd u =\int_0^\infty e^{\eta_u}\dd u  <\infty
 \quad\text{almost surely}.
\eeqlb
From \eqref{eq:critical-bounds}, \eqref{eq:Gi-rates} and \eqref{eq0821}, we have
\beqnn
 \int_0^{\tauY}G_{2,u}^{\theta_2-1}X_u^{\kappa_2}\dd u\le x^{\kappa_2}
 \exp\left(\kappa_2a_1y^{\kappa_1}
 \int_0^\infty G_{2,u}^{\kappa_1}\dd u\right)
 \int_0^\infty G_{2,u}^{\theta_2-1}G_{1,u}^{\kappa_2}\dd u =: x^{\kappa_2} Z_y,
\eeqnn
which is finite almost surely.  If $\tau_0^Y < \infty$, we have $x^{\kappa_2}Z_y \ge \frac{y^{1-\theta_2}}{(1-\theta_2)a_2}$ by \eqref{eq:FK-hitting}.
Notice that the law of \(Z_y\) does not depend on \(x\).  One obtains that
\beqnn
\lim_{x \downarrow 0}\mathbb{P}_{x,y}(\tau_0^Y < \infty) \le \lim_{x \downarrow 0}\mathbb{P}_{x,y}\left(Z_y \ge \frac{y^{1-\theta_2}}{(1-\theta_2)a_2} x^{-\kappa_2}\right) = 0,
\eeqnn 
which proves the first assertion of part (iii).
Assume additionally \(b_{11}+b_{12}>0\), and fix \(T>0\). By \eqref{eq:critical-bounds} and \eqref{eq:FK-hitting}, on the event \(\{\tauY>T\}\), we have
\beqnn 
x^{\kappa_2}\int_0^T
G_{2,u}^{\theta_2-1}G_{1,u}^{\kappa_2}\dd u \le \int_0^{T}G_{2,u}^{\theta_2-1}X_u^{\kappa_2}\dd u
 < \frac{y^{1-\theta_2}}{(1-\theta_2)a_2}.
\eeqnn 
Hence 
\beqnn 
 \left\{x^{\kappa_2}\int_0^T
 G_{2,u}^{\theta_2-1}G_{1,u}^{\kappa_2}\dd u
 \ge\frac{y^{1-\theta_2}}{(1-\theta_2)a_2}\right\}
 \subseteq\{\tauY\le T\}.
\eeqnn 
By Lemma \ref{lem:clock-unbounded-support}, the event on the left has positive probability. Thus the second assertion of part \textup{(iii)} 
holds. Combining it with the first assertion yields the final assertion for
all sufficiently small \(x\). 

	\textup{(iv)}. By part~\textup{(i)}, we have \(\tau_0=\tauY\) almost surely.
	On \(\{\tauY=\infty\}\), \vthirtyfour{\eqref{eq:critical-bounds} and
	\eqref{eq:Gi-rates} show that} \(X_t/G_{1,t}\) is bounded above and below
	by positive finite random constants, uniformly in \(t\ge0\).
	Consequently, it follows from \eqref{eq:Gi-rates} that
	\[
	\lim_{t\to\infty}\frac1t\log X_t=-b_1 
	\qquad \text{and} \qquad 
	\int_0^\infty X_t^{\kappa_2}\dd t<\infty.
	\]
Define the bracket in \eqref{eq:FK-identity} by
\[
 H_t:=y^{1-\theta_2}-(1-\theta_2)a_2
 \int_0^tG_{2,u}^{\theta_2-1}X_u^{\kappa_2}\dd u.
\]
On \(\{\tauY=\infty\}\), \(0<H_t\le y^{1-\theta_2}\). Since $\lim_{t \rightarrow \infty}t^{-1}\log G_{2,t}\to-b_2<0$ by \eqref{eq:Gi-rates}, we have $\lim_{t \rightarrow \infty}G_{2,t} = 0$. It follows that $\lim_{t \rightarrow \infty}Y_t = 0$  by
\eqref{eq:FK-identity}.
Notice that $H_\infty := \lim_{t \rightarrow \infty}H_t \ge0$.  If
\(H_\infty>0\), then \(t^{-1}\log H_t\to0\), and by
\vthirtyfour{\eqref{eq:FK-identity}, we obtain}
$\lim_{t\to\infty}\frac1t\log Y_t = \lim_{t \rightarrow \infty}t^{-1}\log G_{2,t}=-b_2.$
If \(H_\infty=0\), then \(H_t\le1\) for all sufficiently large \(t\).
Consequently,
\[
 \frac1t\log Y_t
 =\frac1t\log G_{2,t}
 +\frac{1}{(1-\theta_2)t}\log H_t
 \le\frac1t\log G_{2,t}
\]
for large $t$, and therefore $
 \limsup_{t\to\infty}\frac1t\log Y_t\le-b_2.$
This proves part~\textup{(iv)}. 
\end{proof}

\subsection[Proof of Corollary~\ref{cor:deterministic-zero-prob}]{\vnew{Proof of Corollary~\ref{cor:deterministic-zero-prob} }}

\begin{proof}[Proof of \vthirtyfour{Corollary~\ref{cor:deterministic-zero-prob}}]
Here \(G_{1,t}=\e^{-b_{10}t}\), \(G_{2,t}=\e^{-b_{20}t}\)  and $G_{2,t}^{\theta_2-1}G_{1,t}^{\kappa_2}
 =\e^{-[\kappa_2b_{10}-(1-\theta_2)b_{20}]t}.$
{
If \(\kappa_2b_{10}-(1-\theta_2)b_{20}\le0\), then the right-hand side
of \eqref{eq:critical-clock-lower} diverges as \(t\to\infty\).
\vthirtyfour{Hence \eqref{eq:FK-hitting} implies} \(\tauY<\infty\).}

Assume \(\kappa_2b_{10}-(1-\theta_2)b_{20}>0\).  By \eqref{eq:FK-identity}, for $t < \tau_0^Y$,
\(Y_t\le y\e^{-b_{20}t}\), and then $
 \int_0^tY_u^{\kappa_1}\dd u
 \le\frac{y^{\kappa_1}}{\kappa_1b_{20}}.$
Using \eqref{eq:critical-X-factor}, we have
\[
 X_t\le x\e^{-b_{10}t}
 \exp\left(\frac{a_1y^{\kappa_1}}{\kappa_1b_{20}}\right), \qquad t < \tau_0^Y.
\]
Consequently, for every \(t<\tauY\), by \eqref{eq:det-survival-condition} and the above,
\[
\int_0^tG_{2,u}^{\theta_2-1}X_u^{\kappa_2}\dd u
\le
\frac{x^{\kappa_2}}
{\kappa_2b_{10}-(1-\theta_2)b_{20}}
\exp\left(
\frac{\kappa_2a_1y^{\kappa_1}}{\kappa_1b_{20}}
\right) < \frac{y^{1-\theta_2}}{(1 - \theta_2)a_2},
\]
which implies that
\beqnn
\int_0^{\tau_0^Y} G_{2,u}^{\theta_2-1}X_u^{\kappa_2}\dd u < \frac{y^{1-\theta_2}}{(1 - \theta_2)a_2}.
\eeqnn 
Hence $\tau_0^Y = \infty$ by \eqref{eq:FK-hitting}.

{
Conversely, suppose for contradiction that \(\tauY=\infty\).
By \eqref{eq:critical-X-factor}, we have
\[
X_t=xe^{-b_{10}t}
\exp\left(a_1\int_0^tY_u^{\kappa_1}\dd u\right),
\qquad t\ge0.
\]
Since \(a_1>0\) and \(Y_u\ge0\), it follows that
\(X_t\ge x\e^{-b_{10}t}\) for every \(t\ge0\), and hence
\[
 \int_0^tG_{2,u}^{\theta_2-1}X_u^{\kappa_2}\dd u
 \ge\frac{x^{\kappa_2}}{\kappa_2b_{10}-(1-\theta_2)b_{20}}
 \left(1-\e^{-[\kappa_2b_{10}-(1-\theta_2)b_{20}]t}\right).
\]
By \eqref{eq:det-extinction-condition}, the limit of the right-hand
side is strictly larger than
\(y^{1-\theta_2}/((1-\theta_2)a_2)\).  \vthirtyfour{This contradicts} \eqref{eq:FK-survival-budget}.  Hence
\(\tauY<\infty\).}
\end{proof}

\appendix

\section{}

\begin{lemma} 
\label{prop:local-wellposedness}
\rev{For every \(x,y>0\), equation~\eqref{eq:noimm-system} admits} a pathwise unique \vthirtyfour{strong}
c\`adl\`ag solution \((X,Y)\), defined on the time interval
\([0, \tau_0\wedge\tau_\infty)\) and taking values in \((0,\infty)^2\).  The solution is adapted to the completed filtration generated by the four driving noises.  The associated stopped interior martingale problem is well-posed and consequently defines a strong Markov family.
\end{lemma}

\begin{proof}
All the coefficients of \eqref{eq:noimm-system} are locally Lipschitz
on \((0,\infty)^2\). Hence, the same localization argument as in
 \cite[Lemma~A.1]{RenXiongYangZhou2022}
gives a pathwise unique \vthirtyfour{strong} solution up to
\(\tau_0\wedge\tau_\infty\). The well-posedness of the stopped
interior martingale problem and the strong Markov property then follow
from the corresponding weak uniqueness.
\end{proof}

The comparison argument below is in the spirit of classical one-dimensional
comparison due to \vthirtyfour{Ikeda and Watanabe}~\cite{IkedaWatanabe1977}, but the infinite-activity jump remainder
is treated explicitly.

\begin{lemma}
	\label{lem:lower-comparison}
	{  Let \((X,Y)\)
	and \(\Xlow\) be the solutions of \eqref{eq:noimm-system} and
	\eqref{eq:lower-envelope}, respectively.  Suppose that
	\(X_0=\Xlow_0=x>0\) and that they are driven by the same
	\(B_1,N_1\).}  Up to their common interior lifetime,
	\begin{equation}
		\Xlow_t\le X_t
		\qquad\text{almost surely}.
		\label{eq:lower-comparison-local}
	\end{equation}
	{If the model \eqref{eq:noimm-system} is
	source-conservative and the solution of \eqref{eq:lower-envelope} is
	global and strictly positive, then}
	\begin{equation}
		\Xlow_t\le X_t,
		\qquad 0\le t<\tau_0,
		\quad\text{almost surely}.
		\label{eq:lower-comparison}
	\end{equation}
	In particular, \(X\) cannot reach zero before \(Y\) on this parameter regime.
\end{lemma}

\begin{proof}
	{
		Let $D_t:=\Xlow_t-X_t.$
		For \(N>1\), define
		\[
		T_N:=\inf\left\{t\ge0:
		X_t\notin(N^{-1},N)\ \text{or}\
		\Xlow_t\notin(N^{-1},N)\ \text{or}\
		Y_t\notin(N^{-1},N)\right\}.
		\]
		Since \(X_0=\Xlow_0=x\), we have \(D_0=0\). For \(t<T_N\),
		\beqnn  
			D_t\ar=\ar -b_{10}\int_0^t
			\bigl(\Xlow_s^{r_{10}}-X_s^{r_{10}}\bigr)\dd s
			-a_1\int_0^tX_s^{\theta_1}Y_s^{\kappa_1}\dd s\cr
			\ar\ar +\int_0^t
			\left(
			\sqrt{2b_{11}\Xlow_s^{r_{11}}}
			-\sqrt{2b_{11}X_s^{r_{11}}}
			\right)\dd B_1(s)\cr
			\ar\ar +\int_0^t\int_0^\infty\int_0^\infty
			z\left(
			\1_{\{u\le b_{12}\Xlow_{s-}^{r_{12}}\}}
			-\1_{\{u\le b_{12}X_{s-}^{r_{12}}\}}
			\right)
			\widetilde N_1(\dd s,\dd z,\dd u).
		\eeqnn
		\vthirtyfour{Applying the Meyer--Tanaka formula to \(u\mapsto u^+\)
		(see Protter}~\cite[Chapter~IV, Section~7]{Protter2005}\vthirtyfour{) and taking expectations, we obtain}
		\beqlb\label{eq:expectation-D-positive}
			\E\bigl[(D_{t\wedge T_N})^+\bigr]
			\ar=\ar -b_{10}\E\left[
			\int_0^{t\wedge T_N}
			\1_{\{D_{s-}>0\}}
			\bigl(\Xlow_s^{r_{10}}-X_s^{r_{10}}\bigr)\dd s
			\right] -a_1\E\left[
			\int_0^{t\wedge T_N}
			\1_{\{D_{s-}>0\}}
			X_s^{\theta_1}Y_s^{\kappa_1}\dd s
			\right]\cr
			\ar\ar +\frac12\E\bigl[L^0_{t\wedge T_N}(D)\bigr]
			+\E\Bigr[\sum_{0<s\le t\wedge T_N}\mathcal R_s\Bigr],
		\eeqlb
		where $\mathcal R_s
	:=
	D_s^+-D_{s-}^+
	-\1_{\{D_{s-}>0\}}\Delta D_s.$
On \(\{D_{s-}>0\}\), we have
		\(\Xlow_{s-}>X_{s-}\). Since \(r_{10}\ge0\), we have $\Xlow_s^{r_{10}}-X_s^{r_{10}}\ge0$,
		and therefore
		\beqnn
			 -b_{10}\E\left[
			\int_0^{t\wedge T_N}
			\1_{\{D_{s-}>0\}}
			\bigl(\Xlow_s^{r_{10}}-X_s^{r_{10}}\bigr)\dd s
			\right] 
			-a_1\E\left[
			\int_0^{t\wedge T_N}
			\1_{\{D_{s-}>0\}}
			X_s^{\theta_1}Y_s^{\kappa_1}\dd s
			\right]
			\le0.
		\eeqnn
		Moreover, let \(M^{D,c}\) be the continuous local martingale part of
		\(D\). Since $u\longmapsto\sqrt{2b_{11}u^{r_{11}}}$
		is Lipschitz on \([N^{-1},N]\), there exists \(C_N>0\) such that
		\[
		\dd\langle M^{D,c}\rangle_s
		=
		\left|
		\sqrt{2b_{11}\Xlow_s^{r_{11}}}
		-\sqrt{2b_{11}X_s^{r_{11}}}
		\right|^2\dd s
		\le C_ND_s^2\dd s
		\]
		for \(s<T_N\). Hence, by the occupation-density formula,
		\beqnn
			L^0_{t\wedge T_N}(D)
			\ar=\ar \lim_{\varepsilon\downarrow0}
			\frac1\varepsilon
			\int_0^{t\wedge T_N}
			\1_{\{0<D_s<\varepsilon\}}
			\dd\langle M^{D,c}\rangle_s\cr
			\ar\le\ar 
			\lim_{\varepsilon\downarrow0}
			\frac{C_N}{\varepsilon}
			\int_0^{t\wedge T_N}
			\1_{\{0<D_s<\varepsilon\}}D_s^2\dd s \le
			\lim_{\varepsilon\downarrow0}C_N\varepsilon t=0.
		\eeqnn
		Therefore, $\E\bigl[L^0_{t\wedge T_N}(D)\bigr]=0.$
		Finally, for a Poisson point \((z,u)\), define
		\[
		\Psi_{z,u}(q)
		:=q+z\1_{\{u\le b_{12}q^{r_{12}}\}},
		\qquad q>0.
		\]
		Since \(r_{12}\ge0\), the map \(q\mapsto\Psi_{z,u}(q)\) is
		nondecreasing. At a jump time \(s\), $D_s
		=\Psi_{z,u}(\Xlow_{s-})-\Psi_{z,u}(X_{s-}).$
		Thus \(D_{s-}>0\) implies \(D_s\ge0\), while \(D_{s-}\le0\) implies
		\(D_s\le0\). Consequently,
		\beqnn
			\mathcal R_s
			 =D_s^+-D_{s-}^+
			-\1_{\{D_{s-}>0\}}\Delta D_s 
			 =\1_{\{D_{s-}>0\}}
			\bigl(D_s-D_{s-}-\Delta D_s\bigr)=0.
		\eeqnn
		Hence
		\[
		 \E\Bigr[\sum_{0<s\le t\wedge T_N}\mathcal R_s\Bigr] =0.
		\]
		Substituting these estimates into
		\eqref{eq:expectation-D-positive}, we obtain $
		\E\bigl[(D_{t\wedge T_N})^+\bigr]
		=0.$
		Therefore, $(D_{t\wedge T_N})^+=0$ almost surely, 
		and hence $\Xlow_{t\wedge T_N}\le X_{t\wedge T_N}$ almost surely.
		Taking a countable intersection over rational \(t\ge0\), using the
		c\`adl\`ag property, and then letting \(N\to\infty\), we obtain
		\eqref{eq:lower-comparison-local}.
		
		If the model \eqref{eq:noimm-system} is source-conservative and
		\(\Xlow\) is global and strictly positive, the preceding comparison
		holds for every \(0\le t<\tau_0\), which proves
		\eqref{eq:lower-comparison}. Since \(\Xlow_t>0\), the inequality
		\(\Xlow_t\le X_t\) also shows that \(X\) cannot reach zero before
		\(Y\).
	}
\end{proof}


\begin{lemma} 
\label{lem:stable-power}
Let \(\alpha\in(1,2)\) and $\mu_\alpha(\dd z):=\frac{\alpha(\alpha-1)}{\Gamma(\alpha)\Gamma(2-\alpha)}z^{-1-\alpha}\1_{\{z>0\}}\dd z$.  For every \(q<\alpha\),  
\[
 \frac1q\int_0^\infty
 \bigl[(1+z)^q-1-qz\bigr]\mu_\alpha(\dd z)
 =j_\alpha(q),
 \qquad q\ne0,
\]
whereas
\[
 \int_0^\infty
 \bigl[\log(1+z)-z\bigr]\mu_\alpha(\dd z)
 =j_\alpha(0)=-1.
\]
Consequently, 
\beqnn
	\int_0^\infty
	\bigl[\Phi_q(u+z)-\Phi_q(u)-z\Phi_q'(u)\bigr]\mu_\alpha(\dd z)
	=j_\alpha(q)u^{q-\alpha} 
	\eeqnn
holds for every \(q<\alpha\) and \(u>0\).
\end{lemma}

\begin{proof}
For \(q\ne0\), let $ g_q(v):=(1+v)^q-1-qv.$
Then $g_q'(v)=q\bigl[(1+v)^{q-1}-1\bigr]$ and $g_q''(v)=q(q-1)(1+v)^{q-2}.$
 \vthirtyfour{Two integrations by parts give}
 \beqnn
 \int_0^\infty g_q(v)v^{-1-\alpha}\dd v
 \ar=\ar \frac1\alpha
   \int_0^\infty g_q'(v)v^{-\alpha}\dd v =\frac1{\alpha(\alpha-1)}
   \int_0^\infty g_q''(v)v^{1-\alpha}\dd v\cr
 \ar=\ar \frac{q(q-1)}{\alpha(\alpha-1)}
   \int_0^\infty
   v^{1-\alpha}(1+v)^{q-2}\dd v =\frac{q(q-1)}{\alpha(\alpha-1)}
   \mathrm B(2-\alpha,\alpha-q),
\eeqnn
which gives the first assertion by \eqref{eq:j-alpha}. 

For \(q=0\), \vthirtyfour{the same two integrations by parts give}
\begin{align*}
 \int_0^\infty[\log(1+v)-v]\mu_\alpha(\dd v)
 &=-\frac1{\Gamma(\alpha)\Gamma(2-\alpha)}
   \int_0^\infty\frac{v^{1-\alpha}}{(1+v)^2}\dd v\\
 &=-\frac1{\Gamma(\alpha)\Gamma(2-\alpha)}
   \mathrm B(2-\alpha,\alpha)
 =-1.
\end{align*}

Finally, the substitution \(z=uv\) gives, for \(q\ne0\),
\beqnn
 \int_0^\infty
 \bigl[\Phi_q(u+z)-\Phi_q(u)-z\Phi_q'(u)\bigr]\mu_\alpha(\dd z) 
  =u^{q-\alpha}\frac1q
  \int_0^\infty
  \bigl[(1+v)^q-1-qv\bigr]\mu_\alpha(\dd v),
\eeqnn
and, for \(q=0\),
\beqnn
  \int_0^\infty
 \bigl[\log(u+z)-\log u-z/u\bigr]\mu_\alpha(\dd z) =u^{-\alpha}
  \int_0^\infty
  \bigl[\log(1+v)-v\bigr]\mu_\alpha(\dd v).
\eeqnn
The result follows.
\end{proof}

\end{document}